\documentclass[a4paper,12pt]{amsart}
\usepackage{amsmath,amsfonts,amssymb}
\usepackage{amsthm}
\usepackage{graphicx}
\usepackage[usenames]{color}

\usepackage{tikz}
\usepackage{mathrsfs}

\usepackage{hyperref}

\newtheorem{theorem}{Theorem}[section]
\newtheorem{definition}[theorem]{Definition}
\newtheorem{prop}[theorem]{Proposition}
\newtheorem{lemma}[theorem]{Lemma}
\newtheorem{cor}[theorem]{Corollary}
\newtheorem{rem}[theorem]{Remark}

\newtheorem{hypothesis}[theorem]{Hypothesis}

\renewcommand{\proof}{{\noindent \bf Proof. \ }}

\newcommand{\eproof}{\hfill \mbox{${\square}$}}

\numberwithin{equation}{section}

\title[\hfil Uniform $L^{\infty}$-Bounds and Upper Semicontinuity of Attractors]
{Uniform $L^{\infty}$-Boundedness and Upper Semicontinuity of Global Attractors for Reaction--Diffusion Equations with Neumann Boundary Conditions on Uniform Jones Domains}

\author[A. L. Pereira]
{Ant\^onio L. Pereira}

\address{Ant\^onio L. Pereira \hfill\break
Instituto de Matem\'atica e Estat\'istica,
Universidade de S\~ao Paulo,
S\~ao Paulo, Brazil}
\email{alpereir@ime.usp.br}

\thanks {Partially supported by FAPESP-SP Brazil grant 2020/14075-6.}

 \subjclass[2020]{Primary: 35B41; Secondary: 35K91, 58D25}
\keywords{Global attractor, Jones domains, uniform $L^{\infty}$ bounds, Mosco convergence, upper semicontinuity, domain perturbation}

\date{\today}

\begin{document}

\begin{abstract}
We consider a family of semilinear parabolic equations with homogeneous Neumann boundary conditions on a family of varying non-smooth domains $\{\Omega_\mu\}_{\mu \in \Lambda} \subset \mathbb{R}^n $. Assuming only that the domains have uniformly bounded volumes, bounded away
from zero, and satisfy a uniform Jones condition, together with the usual structural requirements on the
linear and nonlinear terms, we establish the well-posedness of the problem in an
appropriate scale of fractional Banach spaces and prove the existence of global
attractors. The attractors themselves are shown to be uniformly bounded in $L^{\infty}(\Omega_\mu)$ by an elementary truncation argument, which uses no property of the domains beyond the weak formulation of the Neumann condition. A Moser--Alikakos bootstrap iteration, in conjunction with the uniform Gronwall lemma and the uniform properties of the Jones extension operator, then yields the stronger statement that bounded sets of $L^{q}(\Omega_\mu)$ are absorbed, in finite time and uniformly in $\mu$, into a fixed ball of $L^{\infty}(\Omega_\mu)$; this iteration uses only first-order Sobolev embeddings, which is what makes it available on domains where the domain of the elliptic operator is not contained in $H^{2}$. Finally, assuming the volume convergence of the domains, $|\Omega_\mu \triangle \Omega_0| \to 0$, we construct a framework of connecting maps, show that the uniform Jones condition upgrades this convergence in measure to Mosco convergence of the associated bilinear forms, and prove that the family of attractors is upper semicontinuous at $\mu = 0$ in $L^2(\Omega_\mu)$ and hence, by interpolation, in $H^s(\Omega_\mu)$ for every $s<1$.
\end{abstract}

\maketitle

\section{Introduction}

We consider here a family of semilinear parabolic problems with Neumann  boundary conditions:
\begin{equation} \label{parabolic_problem}
\begin{cases}
u_{t}(x,t) =  \sum_{i,j=1}^{n}\frac{\partial}{\partial x_{i}}\left(a_{ij}(x)\frac{\partial u}{\partial x_{j}}\right)
 + f(u(x,t)), & x\in\Omega_{\mu}, \ t>0, \\
\frac{\partial u}{\partial\nu_{A}}  = 0, & x\in\partial\Omega_{\mu}, \ t>0, \\
u(x,0) = u_{0}(x), & x\in\Omega_{\mu},
\end{cases}
\end{equation}
where the family of open bounded domains $\{\Omega_{\mu}\}_{\mu\in \Lambda} \subset \mathbb{R}^{n}$ and the differential operator
$$
\mathcal{L}u = - \sum_{i,j=1}^{n}\frac{\partial}{\partial x_{i}}\left(a_{ij}(x)\frac{\partial u}{\partial x_{j}}\right)
$$
satisfy rather general geometric and ellipticity conditions, to be made precise
below. We write $A(x)=(a_{ij}(x))$ for the coefficient matrix; the operator used
throughout the paper, denoted $\mathcal{A}_\mu$, is the realisation of
$\mathcal{L}$ on $\Omega_\mu$ with homogeneous Neumann conditions, shifted by a
positive constant so as to be strictly positive (see Section
\ref{sec:preliminaries}). 

 Our main goal is to show  that the  semigroup generated by \eqref{parabolic_problem} is well posed,  dissipative,  the  family of  global attractors  are uniformly bounded in $L^{\infty}( \Omega_{\mu})$ and upper semicontinuous at $\mu=0$. 
 One could also allow the differential operator and the nonlinearity to depend on the parameter $\mu$ but, since  our primary focus  is  on the   challenges arising from the variation of the spatial domain,    we have chosen to omit this dependence in favor of simplicity and readability, as incorporating it introduces no new conceptual difficulties, requiring only standard uniform convergence hypotheses and heavier notation.

The study of semilinear parabolic equations and their long-time behavior is a central topic in the theory of infinite-dimensional dynamical systems. For a wide class of dissipative reaction-diffusion equations, the asymptotic dynamics is  captured by a global attractor. An important  problem in this context is to establish whether these attractors are uniformly bounded in $L^\infty$, ensuring that the physical or biological quantities modeled by the equations do not blow up and allowing for   significant simplification of  the analysis.

The well-posedness of equations similar to the above and the existence of their global attractors under varying  boundary conditions have a rich history in the literature (see, for instance, Carvalho et al. \cite{CarvalhoOlivaPereiraRodriguez1997} and Pereira and Oliva \cite{PereiraOliva2002}). Classical works by Hale and Raugel \cite{HaleRaugel1992}, Arrieta and Carvalho \cite{ArrietaCarvalho2004}, and Daners \cite{Daners2008} have established a solid foundation for understanding how variations in the spatial domain affect the asymptotic dynamics. However, proving uniform $L^\infty$ bounds becomes more challenging when the underlying spatial domains are subject to severe geometric perturbations. The vast majority of existing results rely on the assumption that the family of perturbed domains possesses smooth, or at least uniformly Lipschitz, boundaries satisfying standard cone properties (as addressed in  Arag\~ao, Arrieta and Bruschi \cite{AragaoArrieta2025}, Pereira and Pereira \cite{PereiraPereira2007}, and Barbosa, Pereira, and Pereira \cite{BarbosaPereiraPereira2016}).

One  intention of this paper is to significantly weaken this smoothness assumption,  extending the uniform boundedness of global attractors to a much broader class of non-smooth, highly oscillating, and even fractal domains especially  in the case of homogeneous boundary conditions. The challenge of highly oscillatory boundaries has been recently investigated in the context of continuity of attractors \cite{Pereira_arxiv2024}, as well as problems with terms concentrating on the boundary \cite{AragaoPereiraPereira2014}. However,  when the boundary $\partial\Omega_\mu$ becomes highly irregular as $\mu \to 0$, classical functional tools begin to collapse. Standard Sobolev embeddings may lose their uniformity, and the constants in the trace operator can diverge if the $(n-1)$-dimensional surface measure of the boundary blows up. 

To conduct a rigorous functional analysis independent of these geometric degenerations, we abandon classical smoothness requirements and operate within the topological framework of uniform domains, introduced by P. W. Jones \cite{Jones1981}. The uniform Jones condition prevents the domain from developing topological ``bottlenecks'' or degenerating cusps.

The uniform Jones condition provides not a single estimate but an entire
functional apparatus with $\mu$-independent constants: a uniform bound on the
diameters of the domains, a uniform extension operator, uniform Sobolev and
Gagliardo--Nirenberg inequalities, and a uniform identification of the
fractional power scale with the classical Sobolev spaces up to order one. 
Assembling this apparatus, and then carrying its uniformity through every
subsequent argument, is where the technical burden of this paper lies. The
individual arguments invoked below are classical; what is not routine is that
each must be performed so that no constant is allowed to depend on $\mu$ ---
nor, in the Moser--Alikakos iteration, on the exponent $p$ --- and this
requirement is more delicate than it looks. The interpolation and
Gagliardo--Nirenberg estimates, for instance, require an extension operator
bounded simultaneously on $L^{2}$ and on $H^{1}$, which the construction of
Jones does not directly provide and for which we appeal to the
degree-independent extension of Rogers \cite{Rogers2006}. The reward is
twofold: the optimal critical Sobolev exponent survives uniformly across the
whole family, and the continuity analysis of Section
\ref{sec:upper_semicontinuity} reduces to elementary estimates in $L^{2}$.

It may be useful to say precisely how these results stand in relation to the
closest work on Neumann problems under perturbation of the domain. Arrieta and
Carvalho \cite{ArrietaCarvalho2004} obtain the continuity of attractors on
families of bounded Lipschitz domains, and their standing requirement is the
spectral convergence of the Neumann Laplacian, which they assume and then
characterise.  Their characterisation \cite[Prop.~2.3]{ArrietaCarvalho2004} states in particular that convergence of the domains in measure follows from the spectral convergence, and is therefore necessary for it; it is not, however, listed among the equivalent conditions, and that it is not sufficient is the content of the classical bottleneck examples discussed there. Hypothesis (H2) below is exactly that necessary
condition, and Theorem \ref{thm:mosco} identifies what must be added to it: the
uniform Jones condition, which excludes precisely the topological pinching
through which Neumann spectra are known to degenerate, upgrades convergence in
measure to Mosco convergence of the associated forms, and hence to spectral
convergence. In this sense the geometric hypothesis adopted here is a verifiable
sufficient condition for an assumption that is usually made outright.

Two further comparisons should be recorded, and they run in opposite directions.
Our class of domains contains the uniformly Lipschitz families strictly ---
highly oscillating and fractal boundaries are admitted, and the coefficient
matrix is merely bounded and measurable --- while our conclusion is weaker:
upper semicontinuity in $H^{s}(\Omega_\mu)$ for every $s<1$, against continuity
in $H^{1}$ in \cite{ArrietaCarvalho2004}; the reason is recorded in Remark
\ref{rem:open_s1}. On the other hand, the obstruction that Arrieta and Carvalho
single out as the essential difficulty of Neumann problems --- that
$H^{1}(\Omega)$, unlike $H^{1}_{0}(\Omega)$, admits no extension operator whose
norm is independent of the domain --- is exactly what the uniform Jones
condition removes, and Theorem \ref{thm:extension} is the form in which we use
it. A pattern runs through both comparisons and is worth naming: the classical
arguments are carried out in richer functional settings --- the realisations of the operator on the whole Lebesgue scale in the
$L^{\infty}$ estimates, a fractional power space $X^{1+\epsilon}$ in the comparison
principle --- whereas the arguments used here stay within the $L^{2}$-based
scale and the first-order Sobolev embeddings, which are precisely what the
uniform Jones condition controls uniformly in $\mu$.

It is worth contrasting our geometric framework with recent results dealing with domains that develop true cusps or degenerate into lower-dimensional spaces, such as thin domains. As shown in classical works by Maz'ya and Poborchi \cite{MazyaPoborchi}, and more recently applied to dynamical systems by Tavares-Lima, Lorenzi, and Pereira \cite{Tavares2026}, the presence of true cusps (which inherently violate the Jones condition) leads to a deterioration of the optimal critical Sobolev exponent. In such scenarios, establishing uniform estimates generally requires restricting the growth of the nonlinearity to a suboptimal range and/or relying heavily on the convergence to a one-dimensional limit equation. By contrast, the uniform Jones condition adopted in our work strictly precludes such topological pinching. 
As a result, our framework preserves the optimal critical Sobolev exponent uniformly across the entire family $\{\Omega_\mu\}$.

Our strategy is divided into two steps: we first consider the problem in the base space $(H^{1})'$, to avoid the lack of elliptic regularity in $H^2$, and then use bootstrap tools iteratively to obtain higher regularity. The uniform $L^{\infty}$ bound on the attractors themselves is obtained separately, by a standard truncation argument; what the iteration delivers --- at the cost of tracking the constants carefully, lest they blow up as the number of iterations grows --- is the stronger statement that bounded sets of $L^{q}(\Omega_\mu)$ are absorbed into a fixed ball of $L^{\infty}(\Omega_\mu)$, uniformly in $\mu$.

This method allows us to  apply a  bootstrap iteration in the original $n$-dimensional setting, independent of any dimensional reduction, while simultaneously admitting wildly oscillating and fractal boundaries.

In this work, we focus our analysis on the semilinear parabolic problem under homogeneous Neumann boundary conditions. Under this regime, the boundary integral in the weak formulation cleanly vanishes, allowing the analysis to rely purely on the internal uniform extension properties without imposing strict measure-theoretic constraints on the boundary surface. The uniform Sobolev and Gagliardo--Nirenberg inequalities supplied by the Jones
extension operators are then fed into a Moser--Alikakos bootstrap iteration,
carried out directly on the nonlinear problem, which yields a uniform
$L^{\infty}$ absorbing set up to the optimal critical Sobolev growth. On the
attractors themselves the corresponding bound is elementary and costs no
geometry at all (Proposition \ref{prop:truncation_linfty}), and it is this
bound, together with the rate at which trajectories enter the corresponding
ball, that makes the continuity analysis possible: every orbit entering the
analysis remains in a fixed ball of $L^{\infty}$, on which the nonlinearity is
Lipschitz with a constant independent of $\mu$, and the
upper semicontinuity at $\mu=0$ follows in $L^{2}(\Omega_\mu)$, hence in
$H^{s}(\Omega_\mu)$ for every $s<1$, by elementary estimates.

The paper is organised as follows:
\begin{itemize}
    \item \textbf{Section 2} introduces the geometric framework of Jones uniform domains and assembles the functional toolbox that the uniform condition provides: a two-sided bound on the diameters, the uniform extension theorem, and uniform Sobolev and interpolation estimates. It also states the fundamental hypotheses.
    \item \textbf{Section 3} reformulates the initial boundary value problem as an abstract Cauchy problem within a  suitable  scale of fractional Banach spaces.
    \item \textbf{Section 4} is devoted to proving the local well-posedness of the problem in the fractional scale using the theory of sectorial operators.
    \item \textbf{Sections 5 and 6} establish the global existence of solutions, the dissipativity of the semigroups, and the existence of the global attractors.
    \item \textbf{Section 7} proves the uniform $L^\infty$-boundedness of the global attractors, first by an elementary truncation argument that uses no geometry at all, and then establishes the stronger statement that bounded sets of $L^{q}(\Omega_\mu)$ are absorbed in finite time into a fixed ball of $L^{\infty}(\Omega_\mu)$, by carefully tracking the constants in the Moser--Alikakos iteration.
    \item \textbf{Section 8} focuses on the robustness of the dynamics. Establishing the Mosco convergence of the associated bilinear forms, and hence the strong convergence of resolvents and semigroups, we prove the upper semicontinuity of the global attractors at $\mu=0$ in $L^2(\Omega_\mu)$, and in $H^s(\Omega_\mu)$ for every $s<1$. The uniform $L^\infty$ bound of Section 7 is what makes the nonlinear part of this analysis elementary.
    \item \textbf{Section 9} presents concluding remarks and outlines future perspectives: the endpoint $s=1$ via convergence of spectral structures, intrinsic topological convergence in the Gromov--Hausdorff sense, rates of convergence under exponential attraction, and the transition to boundary-driven dynamics.
\end{itemize}

\section{Preliminaries and Hypotheses}
\label{sec:preliminaries}

In this section, we introduce the geometric framework and the structural assumptions on the nonlinear parabolic problem. We also recall essential classical results, including uniform Sobolev embeddings and the Uniform Gronwall Lemma, which will be the cornerstones of our energy estimates.
\medskip

\noindent\textbf{Notation.}
The spatial dimension is denoted by $n$ throughout, and $\Omega$ denotes a
bounded open subset of $\mathbb{R}^{n}$. For an integer $k\geq 0$ and
$1\le p\le\infty$, $W^{k,p}(\Omega)$ is the usual Sobolev space of functions
whose distributional derivatives up to order $k$ belong to $L^{p}(\Omega)$, and
$H^{k}(\Omega)=W^{k,2}(\Omega)$. Fractional orders occur in this paper only in
the Hilbert case $p=2$; for $s>0$ we set
\begin{equation*}
   H^{s}(\Omega)=\big\{\,U|_{\Omega}\ :\ U\in H^{s}(\mathbb{R}^{n})\,\big\},
   \qquad
   \|u\|_{H^{s}(\Omega)}
   =\inf\big\{\,\|U\|_{H^{s}(\mathbb{R}^{n})}\ :\ U|_{\Omega}=u\,\big\}.
\end{equation*}
For integer $s$ this is consistent with the intrinsic definition whenever
$\Omega$ is a Sobolev extension domain, which is the only case considered in
this paper; all the usual descriptions of $H^{s}(\Omega)$ in fact agree on the
domains considered here, uniformly in the perturbation parameter, as recorded in
Remark \ref{rem:equivalent_Hs} below. The spaces $H^{s}_{0}$, obtained as
closures of $C_{c}^{\infty}(\Omega)$, are never used: the boundary condition in
\eqref{parabolic_problem} is of Neumann type and is encoded in the bilinear
form, not in the space.

For a real number or a real-valued function $\sigma$ we write
$\sigma^{+}=\max\{\sigma,0\}$ for its positive part.
 Throughout, $q:=2^{*}=\frac{2n}{n-2}$ denotes the critical Sobolev exponent.
\medskip

\subsection{Uniform Jones domains}

Intuitively, a Jones domain is one where any two points near each other can be connected by an internal curve that does not deviate too much from the shortest path and avoids coming too close to the boundary (preventing topological ``pinching'' or ``bottlenecks''). It may be extremely rough including, for instance, oscillating boundaries and some fractal structures. This intuitive notion was formalised by P.W. Jones in \cite{Jones1981}.
\begin{definition}[Jones domain] {Let $\delta_0 > 0$ and $\epsilon \ge 1$. A domain $D \subset \mathbb{R}^n $ is called a $(\delta_0, \epsilon)$-Jones  domain if, for any pair of points $x, y \in D$ such that $|x - y| \le \delta_0$, there exists a rectifiable curve $\gamma \subset D$ connecting $x$ and $y$ satisfying:}
\begin{enumerate}
    \item \textit{$\ell(\gamma) \le \epsilon |x - y|$, where $\ell(\gamma)$ denotes the arc length of $\gamma$.}
    \item \textit{$\text{dist}(z, \partial D) \ge \frac{1}{\epsilon} \frac{|x - z||y - z|}{|x - y|}$, for all $z \in \gamma$.}
\end{enumerate}
  We  will simply say $\Omega $ is a Jones domain, if it is  a $(\delta_0, \epsilon)$-Jones  domain, for some  $\delta>0 $ and $\epsilon >1$. 
 \end{definition}

Let $\{\Omega_\mu\}_{\mu \in \Lambda}$ be a family of bounded domains in $\mathbb{R}^n $ ($n \geq 2$).

  \begin{definition}[Uniform Jones family ] \label{def:uniform_Jones}
  We say that  $\{\Omega_\mu\}_{\mu \in \Lambda}$  is a  uniform Jones  family of domains  if there exist strictly positive constants $\delta_0$, $\epsilon$ and $d_{*}$, independent of the perturbation parameter $\mu$, such that $\Omega_\mu$ is a $(\delta_0, \epsilon)$-Jones domain and $\operatorname{diam}(\Omega_\mu)\ge d_{*}$ for all $\mu \in (0, \mu_0]$.
  \end{definition}

\begin{lemma}[Uniform bound on the diameter]\label{lem:uniform_diameter}
Let $\Omega\subset\mathbb{R}^{n}$ be a $(\delta_0,\epsilon)$-Jones domain with
$|\Omega|\le V_0<\infty$. Put
\begin{equation}\label{eq:diam_constants}
    c:=\frac{\delta_0}{4\epsilon},
    \qquad \rho:=\delta_0+2c,
    \qquad N_0:=\Big\lfloor\frac{V_0}{\omega_n c^{\,n}}\Big\rfloor ,
\end{equation}
where $\omega_n$ denotes the volume of the unit ball of $\mathbb{R}^{n}$. Then
\begin{equation}\label{eq:diam_bound}
    \operatorname{diam}(\Omega)\ \le\ \max\big\{\,2\delta_0,\ 2\rho N_0\,\big\} .
\end{equation}
In particular, on a uniform Jones family with uniformly bounded volume the
diameters are bounded by a constant depending only on $\delta_0$, $\epsilon$,
$n$ and $V_0$.
\end{lemma}

\proof
If there is $x_0\in\Omega$ with $\Omega\subset B(x_0,\delta_0)$, then
$\operatorname{diam}(\Omega)\le2\delta_0$ and \eqref{eq:diam_bound} holds. We may
therefore assume that $\Omega\not\subset B(x,\delta_0)$ for every $x\in\Omega$.

\textbf{Step 1: a uniform interior ball near every point.} We claim that for
every $x\in\Omega$ there exists $z_x\in\Omega$ with
\begin{equation}\label{eq:interior_ball}
    |x-z_x|=\frac{\delta_0}{2}
    \qquad\text{and}\qquad
    B(z_x,c)\subset\Omega .
\end{equation}
Indeed, $\Omega$ is open and connected, hence path-connected, and is not
contained in $B(x,\delta_0)$; along a path joining $x$ to a point outside that
ball the continuous function $|x-\cdot|$ attains the value $\delta_0$, so there
is $y\in\Omega$ with $|x-y|=\delta_0$. Let $\gamma\subset\Omega$ be a curve
associated with the pair $(x,y)$ by the definition of a
$(\delta_0,\epsilon)$-Jones domain, and let $z_x\in\gamma$ satisfy
$|x-z_x|=\delta_0/2$, which exists by continuity. Then
$|y-z_x|\ge|x-y|-|x-z_x|=\delta_0/2$, and condition (2) of that definition gives
\begin{equation*}
    \operatorname{dist}(z_x,\partial\Omega)
    \ \ge\ \frac{1}{\epsilon}\,\frac{|x-z_x|\,|y-z_x|}{|x-y|}
    \ \ge\ \frac{1}{\epsilon}\,\frac{(\delta_0/2)^{2}}{\delta_0}
    \ =\ c ,
\end{equation*}
which is \eqref{eq:interior_ball}. We stress that the ball is centred at $z_x$
and not at $x$: since $\epsilon\ge1$ one has $c\le\delta_0/4<\delta_0/2$, so in
general $x\notin B(z_x,c)$.

\textbf{Step 2: counting.} Let $\{p_1,\dots,p_N\}\subset\Omega$ be
$\rho$-separated, that is, $|p_i-p_j|\ge\rho$ for $i\neq j$. By
\eqref{eq:interior_ball},
\begin{equation*}
    |z_{p_i}-z_{p_j}|\ \ge\ |p_i-p_j|-|p_i-z_{p_i}|-|p_j-z_{p_j}|
    \ \ge\ \rho-\delta_0\ =\ 2c ,
\end{equation*}
so the balls $B(z_{p_i},c)$ are pairwise disjoint and contained in $\Omega$;
hence $N\,\omega_n c^{\,n}\le|\Omega|\le V_0$ and $N\le N_0$. Every
$\rho$-separated subset of $\Omega$ is therefore finite with at most $N_0$
elements, so a maximal one exists, and by maximality it is a $\rho$-net:
$\Omega$ is covered by at most $N_0$ balls of radius $\rho$.

\textbf{Step 3: chaining.} A connected set covered by $N_0$ balls of radius
$\rho$ has diameter at most $2\rho N_0$. Indeed, given $x,y\in\Omega$,
connectedness forces the existence of a chain of at most $N_0$ of these balls,
consecutive ones having a common point, whose union contains both $x$ and $y$;
the diameter of such a union is at most $2\rho N_0$.
\eproof

\begin{rem}\label{rem:holdall}
Lemma \ref{lem:uniform_diameter} is what prevents a uniform Jones family of
uniformly bounded volume from developing long thin tentacles: the condition
carries a quantitative interior thickness, and a domain of small volume cannot
be long. It is used in Section \ref{sec:upper_semicontinuity} to place the whole
family inside a fixed bounded region, without assuming it.
\end{rem}

\begin{rem}[Normalisation of the scale]\label{rem:scale_normalisation}
The Jones condition constrains the shape of a domain, not its size: for
$\lambda>0$ the dilate $\lambda\Omega$ of a $(\delta_{0},\epsilon)$-Jones domain
is a $(\lambda\delta_{0},\epsilon)$-Jones domain, with the \emph{same}
$\epsilon$, and a ball of radius $r$ is a $(\delta_{0},\epsilon)$-Jones domain
with constants independent of $r$.
This is why Definition \ref{def:uniform_Jones} includes a uniform lower bound
$\operatorname{diam}(\Omega_\mu)\ge d_{*}$. The requirement is mild: whenever
the family has uniformly bounded volume it suffices to assume
$\inf_{\mu\in\Lambda}|\Omega_\mu|\ge v_{0}>0$, which gives
$\operatorname{diam}(\Omega_\mu)\ge 2(v_{0}/\omega_{n})^{1/n}$; and in the
situation of Section \ref{sec:upper_semicontinuity} hypothesis (H2) gives
$|\Omega_\mu|\to|\Omega_{0}|>0$, so that this holds automatically for all
sufficiently small $\mu$.

\end{rem}

The fundamental consequence of the uniform Jones condition is the existence of a universal extension operator in Sobolev spaces $W^{k,p}$, independent of $\mu \in \Lambda$.

\begin{theorem}[Uniform Extension Theorem]\label{thm:extension} Suppose  that  $\{\Omega_\mu\}_{\mu \in \Lambda}$  is a  uniform Jones  family of domains. Then 
there exists, for each $\mu$, a linear operator $\mathfrak{E}_\mu$ with
$\mathfrak{E}_\mu u|_{\Omega_\mu}=u$ which is bounded from
$W^{k,p}(\Omega_\mu)$ into $W^{k,p}(\mathbb{R}^{n})$ simultaneously for all
integers $k\ge 0$ and all $1\le p\le\infty$. For each fixed pair $(k,p)$ its
operator norm is bounded by a constant $K=K(n,k,p,\delta_{0},\epsilon)$,
independent of $\mu$:
\begin{equation}
    \| \mathfrak{E}_\mu u\|_{W^{k,p}(\mathbb{R}^n)} \leq K \|u\|_{W^{k,p}(\Omega_\mu)}, \quad \forall u \in W^{k,p}(\Omega_\mu).
\end{equation}
\end{theorem}

\proof
The result was proved by Jones in \cite{Jones1981}, where it is also shown that
this condition is optimal in a certain sense (see Theorems 1 and 3 in
\cite{Jones1981}); the bounds obtained there depend on the domain only through
the pair $(\delta_{0},\epsilon)$, so they are uniform on a uniform Jones family.
Moreover, by a result of Rogers \cite{Rogers2006}, a single linear operator
$\mathfrak{E}_{\mu}$ may be chosen which is bounded simultaneously on the whole
scale $W^{k,p}(\Omega_{\mu})$, $k\ge 0$, $1\le p\le\infty$ --- in particular on
$L^{2}(\Omega_{\mu})$ and on $H^{1}(\Omega_{\mu})$ at once --- again with bounds
depending only on $(\delta_{0},\epsilon)$, $n$, $k$ and $p$. Throughout this
paper $\mathfrak{E}_{\mu}$ denotes such a degree-independent extension operator.

One point in the quoted results requires a word, precisely because it is the
sort of hidden dependence on the domain that this paper is about. The theorem of
\cite{Rogers2006} is proved under the normalisation
$\operatorname{diam}(\Omega)\ge1$, and its constant is allowed to grow without
bound as the diameter tends to $0$ with the remaining parameters fixed; this is
the only place where the size, rather than the shape, of the domain enters.The diameters of our family are bounded below by
$d_{*}>0$, so it suffices to apply the theorem to the dilated family
$\lambda\Omega_\mu$ with the \emph{single} factor
$\lambda:=\max\{1,1/d_{*}\}$, chosen once and for all. Indeed
$\{\lambda\Omega_\mu\}$ is again a uniform Jones family, with parameters
$(\lambda\delta_{0},\epsilon)$, and each of its members has diameter at least
$1$; moreover $u\mapsto u(\lambda^{-1}\,\cdot\,)$ is an isomorphism of
$W^{k,p}(\Omega_\mu)$ onto $W^{k,p}(\lambda\Omega_\mu)$ whose norm and inverse
norm depend only on $\lambda$, $n$, $k$ and $p$. Conjugating the operator of
\cite{Rogers2006} by this dilation therefore yields $\mathfrak{E}_{\mu}$ with a
bound independent of $\mu$. (We note also that the constant $\epsilon$ of
Definition \ref{def:uniform_Jones} is the reciprocal of the one used in
\cite{Jones1981} and \cite{Rogers2006}.)

Only three indices are ever used, namely $(k,p)=(0,2)$, $(1,2)$ and
$(0,\infty)$ --- that is, boundedness on $L^{2}$, on $H^{1}$ and on
$L^{\infty}$ --- and from now on $K$ denotes a single constant dominating these
three operator norms. We stress that, although the constant furnished by
Theorem \ref{thm:extension} does depend on $k$ and on $p$, no limit is ever
taken in these indices. In particular the exponents $p_{k}\to\infty$ of the
Moser--Alikakos iteration of Section \ref{sec:uniform_linfty} are Lebesgue
exponents of the iterated functions $|u|^{p_{k}/2}$, which are inserted into the
\emph{fixed} Gagliardo--Nirenberg inequality associated with
$H^{1}(\Omega_\mu)\hookrightarrow L^{2\chi}(\Omega_\mu)$; the extension operator
itself is applied only with the three indices listed above, and the constant
$C_{GN}$ of Step 2 of Theorem \ref{thm:uniform_linfty} is fixed once and for all
before the iteration begins.
\eproof

Combining Theorem \ref{thm:extension} with the standard Sobolev embeddings in $\mathbb{R}^n $, we obtain uniform embedding constants for our family of domains.

\begin{lemma}[Uniform Sobolev Embedding]\label{lem:uniform_sobolev}   Suppose  that  $\{\Omega_\mu\}_{\mu \in \Lambda}$  is a  uniform Jones  family of domains, and fix an integer $k\ge1$ and $1\le p<\infty$ with $kp<n$. Then there exists a constant $C_S=C_S(n,k,p,K) > 0$, independent of $\mu$, such that the continuous embedding $W^{k,p}(\Omega_\mu) \hookrightarrow L^{\frac{pn}{n-kp}}(\Omega_\mu)$, for $kp < n$, holds uniformly:
\begin{equation}
    \|u\|_{L^{\frac{pn}{n-kp}}(\Omega_\mu)} \leq C_S \|u\|_{W^{k,p}(\Omega_\mu)}, \quad \forall u \in W^{k,p}(\Omega_\mu).
\end{equation}
\end{lemma}

The uniform extension property also makes the interpolation scale of the family
uniform.

\begin{lemma}[Uniform Interpolation]\label{lem:uniform_interpolation}
Suppose that $\{\Omega_\mu\}_{\mu\in\Lambda}$ is a uniform Jones family of
domains. Then, for every $\theta\in[0,1]$,
\begin{equation}\label{eq:uniform_interpolation}
   \big[\,L^{2}(\Omega_\mu),\,H^{1}(\Omega_\mu)\,\big]_{\theta}
   \;=\; H^{\theta}(\Omega_\mu),
\end{equation}
with norm equivalence constants depending only on $n$ and on the uniform
extension bound $K$ of Theorem \ref{thm:extension}, and therefore independent of
$\mu$.
\end{lemma}

\proof
Write $R_\mu U=U|_{\Omega_\mu}$ for the restriction operator, and recall the
classical identity
$[L^{2}(\mathbb{R}^{n}),H^{1}(\mathbb{R}^{n})]_{\theta}=H^{\theta}(\mathbb{R}^{n})$,
immediate from the Fourier characterisation of $H^{s}(\mathbb{R}^{n})$.

Since $\|R_\mu\|_{\mathcal{L}(L^{2}(\mathbb{R}^{n}),L^{2}(\Omega_\mu))}\le 1$ and
$\|R_\mu\|_{\mathcal{L}(H^{1}(\mathbb{R}^{n}),H^{1}(\Omega_\mu))}\le 1$, the
interpolation property of the complex method gives
$\|R_\mu U\|_{[L^{2}(\Omega_\mu),H^{1}(\Omega_\mu)]_{\theta}}
\le\|U\|_{H^{\theta}(\mathbb{R}^{n})}$ for every
$U\in H^{\theta}(\mathbb{R}^{n})$; hence
$H^{\theta}(\Omega_\mu)\hookrightarrow[L^{2}(\Omega_\mu),H^{1}(\Omega_\mu)]_{\theta}$
with constant $1$. We stress that this inclusion requires no hypothesis
whatsoever on $\Omega_\mu$.

The converse inclusion is where the geometry enters. By Theorem
\ref{thm:extension}, $\mathfrak{E}_\mu$ is bounded on both endpoints of the
couple, with
$\max\{\|\mathfrak{E}_\mu\|_{\mathcal{L}(L^{2})},
\|\mathfrak{E}_\mu\|_{\mathcal{L}(H^{1})}\}\le K$ uniformly in $\mu$.
Interpolating, $\mathfrak{E}_\mu$ maps
$[L^{2}(\Omega_\mu),H^{1}(\Omega_\mu)]_{\theta}$ boundedly into
$[L^{2}(\mathbb{R}^{n}),H^{1}(\mathbb{R}^{n})]_{\theta}=H^{\theta}(\mathbb{R}^{n})$,
with norm at most $K$. Since $R_\mu\mathfrak{E}_\mu=I$, every
$u\in[L^{2}(\Omega_\mu),H^{1}(\Omega_\mu)]_{\theta}$ is the restriction of
$\mathfrak{E}_\mu u\in H^{\theta}(\mathbb{R}^{n})$, so that
$u\in H^{\theta}(\Omega_\mu)$ with
$\|u\|_{H^{\theta}(\Omega_\mu)}\le K\,\|u\|_{[L^{2},H^{1}]_{\theta}}$.
\eproof

\begin{cor}[Uniform Sobolev embedding of fractional order]\label{cor:uniform_sobolev_frac}
Suppose that $\{\Omega_\mu\}_{\mu\in\Lambda}$ is a uniform Jones family of
domains, and let $0\le s\le 1$ with $2s<n$. Then there exists a constant
$C_S'>0$, independent of $\mu$, such that
\begin{equation}\label{eq:uniform_sobolev_frac}
    \|u\|_{L^{\frac{2n}{n-2s}}(\Omega_\mu)}
    \le C_S'\,\|u\|_{H^{s}(\Omega_\mu)},
    \qquad u\in H^{s}(\Omega_\mu).
\end{equation}
The same holds, with a constant again independent of $\mu$, when
$H^{s}(\Omega_\mu)$ is replaced by the interpolation space
$[\,L^{2}(\Omega_\mu),H^{1}(\Omega_\mu)\,]_{s}$.
\end{cor}

\proof
Let $u\in H^{s}(\Omega_\mu)$ and let $U\in H^{s}(\mathbb{R}^{n})$ be any
extension of $u$. By the Sobolev embedding
$H^{s}(\mathbb{R}^{n})\hookrightarrow L^{\frac{2n}{n-2s}}(\mathbb{R}^{n})$,
whose constant depends only on $n$ and $s$,
\begin{equation*}
    \|u\|_{L^{\frac{2n}{n-2s}}(\Omega_\mu)}
    \le\|U\|_{L^{\frac{2n}{n-2s}}(\mathbb{R}^{n})}
    \le C\,\|U\|_{H^{s}(\mathbb{R}^{n})} .
\end{equation*}
Taking the infimum over all extensions $U$ of $u$ gives
\eqref{eq:uniform_sobolev_frac}, the norm of $H^{s}(\Omega_\mu)$ being the
quotient norm fixed in the Notation; note that no property of $\Omega_\mu$ has
been used so far. The last assertion follows from Lemma
\ref{lem:uniform_interpolation}, which identifies the two norms up to constants
depending only on $n$ and $K$; it is here, and only here, that the uniform Jones
condition is needed.
\eproof

\begin{rem}
For $s=1$ Corollary \ref{cor:uniform_sobolev_frac} reduces to the case $k=1$,
$p=2$ of Lemma \ref{lem:uniform_sobolev}. We isolate the fractional case because
it is the one required in Section \ref{sec:local_existence}, where the relevant space is
$X_{1-\delta}$ rather than $H^{1}(\Omega_\mu)$.
\end{rem}

\begin{rem}[Equivalent descriptions of $H^{s}(\Omega_\mu)$]\label{rem:equivalent_Hs}
On a uniform Jones family the choice of definition is immaterial. For
$0<s<1$ and every $\mu$, the space $H^{s}(\Omega_\mu)$ fixed in the Notation
above coincides, with norms equivalent up to constants depending only on $n$,
$s$ and $K$, with each of the following:
\begin{enumerate}
 \item[(i)] the Sobolev--Slobodeckij space of those $u\in L^{2}(\Omega_\mu)$ for
 which
 \[
   \iint_{\Omega_\mu\times\Omega_\mu}
   \frac{|u(x)-u(y)|^{2}}{|x-y|^{\,n+2s}}\,dx\,dy<\infty ;
 \]
 \item[(ii)] the interpolation space
 $[\,L^{2}(\Omega_\mu),H^{1}(\Omega_\mu)\,]_{s}$ of Lemma
 \ref{lem:uniform_interpolation};
 \item[(iii)] the completion of $C^{\infty}(\overline{\Omega_\mu})$ in any of
 the corresponding norms.
\end{enumerate}
Indeed, (iii) follows by extending $u$ to $U\in H^{s}(\mathbb{R}^{n})$,
mollifying and restricting; and (i)--(ii) follow from the
retraction--coretraction argument of Lemma \ref{lem:uniform_interpolation}, the
corresponding identities on $\mathbb{R}^{n}$ being classical. On a general
bounded open set none of these equivalences need hold: it is precisely the
uniform extension property that makes the fractional scale canonical, and
uniformly so in $\mu$.
 \end{rem}

\subsection{Hypotheses on the Nonlinearity}
\label{subsec:hyp_nonlinearity}

We  will need  that the nonlinear reaction term $f: \mathbb{R} \to \mathbb{R}$ satisfies  standard conditions in order to guarantee local well-posedness, global existence, and the dissipativity of the dynamical system.

\begin{hypothesis}[Lipschitz growth condition (subcritical)  on $f$]\label{hyp:nonlinear_growth}    There exist constants $C > 0$ and $p \geq 1$ such that for all $x, y \in \mathbb{R}$,
    \begin{equation*}
        |f(x) - f(y)| \leq C|x - y|(1 + |x|^{p-1} + |y|^{p-1}).
    \end{equation*}
    Setting $y=0$ immediately yields the standard polynomial growth bound $|f(s)| \leq C'(1 + |s|^p)$. If $n \geq 3$, we require the subcriticality condition $p < \frac{n+2}{n-2}$. If $n \leq 2$, the exponent $p$ can be an arbitrary finite number.
 \end{hypothesis}
 
 \begin{hypothesis}[Dissipativity condition  on $f$]\label{hyp:nonlinear_dissipativity}   
 There exist constants $\beta > 0$ and $C_f \geq 0$ such that
    \begin{equation*}
            f(s)s \leq -\beta s^2 + C_f, \quad \forall s \in \mathbb{R}.
                \end{equation*}
 \end{hypothesis}

The Lipschitz growth condition ensures that the Nemytskii operator associated with $f$ maps the phase space into the base dual space smoothly enough to guarantee local well-posedness. The dissipativity condition then provides the necessary absorbing sets to ensure global existence and the existence of a global attractor in the energy space.

\subsection{The Elliptic Operator and Functional Framework}
\label{subsec:hyp_elliptic}

We consider a second-order elliptic operator in divergence form, initially denoted by $\mathcal{L} u = \mathrm{div}(A \nabla u)$, subject to homogeneous Neumann conditions and satisfying the following strict ellipticity condition.

\begin{hypothesis}[Ellipticity and Coefficients]\label{hyp:elliptic_condition}
The matrix $A \in L^\infty(\mathbb{R}^n, \mathbb{R}^{n \times n})$ is symmetric and satisfies a strict ellipticity condition: there exists $\alpha > 0$ such that
\begin{equation}
    A \xi \cdot \xi \geq \alpha |\xi|^2, \quad \text{for a.e. } x \in \mathbb{R}^n \text{ and all } \xi \in \mathbb{R}^n.
\end{equation}
\end{hypothesis}

To ensure that the linear operator associated with our problem is strictly positive (and therefore  allowing the construction of a fractional power scale), we introduce a standard spectral shift. Utilising the dissipativity constant $\beta > 0$ established in Hypothesis \ref{hyp:nonlinear_dissipativity}, we define the shifted linear operator $\mathcal{A}_{\mu}: H^1(\Omega_\mu) \to (H^1(\Omega_\mu))'$ by:
\begin{equation*}
    \mathcal{A}_{\mu}u = -\mathrm{div}(A(x)\nabla u) + \frac{\beta}{2}u.
\end{equation*}

The associated symmetric bilinear form $a_\mu: H^1(\Omega_\mu) \times H^1(\Omega_\mu) \to \mathbb{R}$, given by
\begin{equation}
    a_\mu(u,v) = \int_{\Omega_\mu} \left( A(x)\nabla u \cdot \nabla v + \frac{\beta}{2} u v \right) dx,
\end{equation}
is strictly coercive on $H^1(\Omega_\mu)$; more precisely,
$a_\mu(u,u)\ge\alpha_0\|u\|^2_{H^1(\Omega_\mu)}$ with
$\alpha_0:=\min\{\alpha,\beta/2\}$, a constant independent of $\mu$. This coercivity yields the standard Gelfand triple $H^1(\Omega_\mu) \hookrightarrow L^2(\Omega_\mu) \hookrightarrow (H^1(\Omega_\mu))'$ and provides a family of positive, self-adjoint operators with compact resolvents.

\textbf{Reformulation of the Problem.} With the spectral shift applied to the linear operator, the original parabolic equation $u_t - \mathrm{div}(A(x)\nabla u) = f(u)$ is naturally rewritten as:
\begin{equation}
    u_t + \mathcal{A}_{\mu}u = \tilde{f}(u),
\end{equation}
where the modified nonlinearity is given by $\tilde{f}(s) = f(s) + \frac{\beta}{2}s$. It is straightforward to check that $\tilde{f}$ inherits the growth and dissipativity from $f$. 
Thus, the modified dynamical system maintains the exact same essential structure while ensuring the strict positivity of the principal linear part.

\subsection{Global Attractors and Invariant Trajectories}

To study the long-time dynamics of our system, we recall the standard framework of infinite-dimensional dynamical systems. Let $(X, d)$ be a complete metric space and let $\{S(t)\}_{t \ge 0}$ be a strongly continuous semigroup acting on $X$.

\begin{definition}[Global Attractor]
A set $\mathscr{A} \subset X$ is called a global attractor for the semigroup $\{S(t)\}_{t \ge 0}$ if it satisfies the following properties:
\begin{itemize}
    \item[(i)] $\mathscr{A}$ is compact in $X$.
    \item[(ii)] $\mathscr{A}$ is strictly invariant, that is, $S(t)\mathscr{A} = \mathscr{A}$ for all $t \ge 0$.
    \item[(iii)] $\mathscr{A}$ attracts all bounded subsets of $X$, meaning that for any bounded set $B \subset X$,
    \begin{equation}
        \lim_{t \to \infty} \mathrm{dist}_X(S(t)B, \mathscr{A}) = 0,
    \end{equation}
    where $\mathrm{dist}_X(A, B) = \sup_{a \in A} \inf_{b \in B} d(a,b)$ denotes the Hausdorff semidistance in $X$.
\end{itemize}
\end{definition}

A fundamental consequence of the strict invariance property (ii) is the existence of complete trajectories. If $u_0 \in \mathscr{A}$, there exists a global trajectory $u: \mathbb{R} \to \mathscr{A}$ such that $u(0) = u_0$ and $S(t)u(\tau) = u(t+\tau)$ for all $\tau \in \mathbb{R}$ and $t \ge 0$. This property allows us to evaluate the flow backwards in time, which is a critical tool when performing differential energy estimates uniformly on the attractor without encountering initial-time singularities.

\subsection{The Uniform Gronwall Lemma}

To propagate $L^p$ energy bounds across the scale of $L^p$ spaces, we will use the following tool from the theory of dynamical systems (see \cite{foias_prodi_1967} or \cite{temam1988} for the proof).

\begin{lemma}[Uniform Gronwall Lemma]\label{lem:gronwall}
Let $g, h$, and $y$ be non-negative locally integrable functions on $[t_0, \infty)$ such that
\begin{equation}
    \frac{dy}{dt} \leq g y + h \quad \text{for } t \geq t_0,
\end{equation}
in the sense of distributions. Assume that there exist positive constants $r, a_1, a_2$, and $a_3$ such that
\begin{equation}
    \int_t^{t+r} g(s) ds \leq a_1, \quad \int_t^{t+r} h(s) ds \leq a_2, \quad \int_t^{t+r} y(s) ds \leq a_3,
\end{equation}
for all $t \geq t_0$. Then, for any $t \geq t_0$, the function $y$ satisfies the uniform bound:
\begin{equation}
    y(t+r) \leq \left( \frac{a_3}{r} + a_2 \right) e^{a_1}.
\end{equation}
\end{lemma}

\section{Abstract Formulation and Fractional Spaces}

To establish a rigorous framework for our problem on the possibly rough Jones domain $\Omega_\mu$, we cast the equation as an abstract Cauchy problem, as indicated in Section \ref{subsec:hyp_elliptic}. We rely on the theory of sectorial operators and interpolation spaces, using a Gelfand triple to deal with the lack of higher-order spatial regularity.

Let $V = H^1(\Omega_\mu)$ and $H = L^2(\Omega_\mu)$. Identifying $H$ with its dual $H'$, we obtain the standard Gelfand triple $V \hookrightarrow H \hookrightarrow V'$, where $V' = (H^1(\Omega_\mu))'$ and the embeddings are continuous and dense (see, for instance, \cite{Yagi2010} for details). 

We choose our primary base space to be $X_0 = V' = (H^1(\Omega_\mu))'$ and define the linear operator $\mathcal{A}_{\mu} \in \mathcal{L}(V, V')$ associated with the shifted bilinear form $a_\mu$ introduced in Section \ref{sec:preliminaries}:
\begin{equation}  \label{def:lin_operator}
    \langle \mathcal{A}_{\mu} u, v \rangle_{V', V} = a_\mu(u, v) = \int_{\Omega_\mu} \left( A(x)\nabla u \cdot \nabla v + \frac{\beta}{2} uv \right) dx, \quad \forall u, v \in V.
\end{equation}

 As established previously, the bilinear form  is symmetric, continuous, and strictly coercive on $V$. Consequently, $\mathcal{A}_{\mu}$ is a densely defined, strictly positive, self-adjoint, therefore sectorial operator in the Hilbert space $X_0$, with domain $D(\mathcal{A}_{\mu}) = V = H^1(\Omega_\mu)$. Therefore, $\mathcal{A}_{\mu}$ generates an analytic semigroup on $X_0$.

 \begin{rem}\label{rem:uniform_semigroup}
The constants in the analytic semigroup estimates for $\mathcal{A}_\mu$ are
likewise independent of $\mu$, but for a reason quite different from the one
behind \eqref{eq:sobolev_identification}: here no property of $\Omega_\mu$ is
involved at all. Since $\mathcal{A}_\mu$ is self-adjoint and bounded below by
$\beta>0$, the spectral theorem gives
\begin{equation}\label{eq:uniform_semigroup}
    \big\|\mathcal{A}_\mu^{\gamma}e^{-\mathcal{A}_\mu t}\big\|
    \ \le\ \sup_{\lambda\ge\beta}\lambda^{\gamma}e^{-\lambda t}
    \ \le\ \Big(\frac{\gamma}{e}\Big)^{\gamma}\,t^{-\gamma},
    \qquad \gamma\ge0,\ t>0,
\end{equation}
and hence, since $\|v\|_{X_{\sigma}}=\|\mathcal{A}_\mu^{\sigma}v\|_{X_{0}}$,
\begin{equation*}
    \big\|e^{-\mathcal{A}_\mu t}\big\|_{\mathcal{L}(X_{\sigma},X_{\sigma+\gamma})}
    \ \le\ M_{\gamma}\,t^{-\gamma},
\end{equation*}
with $M_{\gamma}$ depending on $\gamma$ alone. The constants $M_{1/2}$ and
$M_{\delta}$ used in Section \ref{sec:global_attractor_section} are of this
kind.
\end{rem}

 For the sake of a less cumbersome notation we will drop the subscript $\mu $, and denote  this operator simply by $\mathcal{A}$,  wherever there is no danger of confusion.

We can define the fractional power operators associated with $\mathcal{A} $, denoted by ${\mathcal{A}}^\alpha$, and their corresponding fractional power spaces $X_\alpha = D({\mathcal{A}}^\alpha)$ for $\alpha \geq 0$, equipped with the graph norm $\|u\|_{X_\alpha} = \|{\mathcal{A}}^\alpha u\|_{X_0}$ (see \cite{Henry1981}, \cite{Yagi2010} for details). By definition, we have:
\begin{equation}
    X_0 = (H^1(\Omega_\mu))', \quad \text{and} \quad X_1 = H^1(\Omega_\mu).
\end{equation}

To better understand the properties of the spaces $X_\alpha$ for $\alpha > 1$ (including the compactness of their embeddings), we introduce an auxiliary scale of spaces based on $L^2(\Omega_\mu)$. Let $Y_0 = H = L^2(\Omega_\mu)$, and let $\tilde{\mathcal{A}}$ be the restriction of $\mathcal{A} $ to $Y_0$. Then,  $\tilde{\mathcal{A}}$ is the unbounded operator on $L^2(\Omega_\mu)$ with domain $D(\tilde{\mathcal{A}}) = \{ u \in H^1(\Omega_\mu) : \mathcal{A} u \in L^2(\Omega_\mu) \}$.

Like $\mathcal{A}$, the operator $\tilde{\mathcal{A}}$ is strictly positive,
self-adjoint and sectorial on $Y_{0}$, and we define its fractional power spaces
by $Y_\gamma=D(\tilde{\mathcal{A}}^{\gamma})$ for $\gamma\ge 0$. Since
$\tilde{\mathcal{A}}$ is strictly positive, the definition extends to negative
exponents: for $\gamma<0$ we let $Y_\gamma$ be the completion of $Y_{0}$ under
$\|u\|_{Y_\gamma}=\|\tilde{\mathcal{A}}^{\gamma}u\|_{Y_{0}}$, so that
$\{Y_\gamma\}_{\gamma\in\mathbb{R}}$ is a scale of Hilbert spaces with
$Y_\gamma\hookrightarrow Y_{\gamma'}$ continuously and densely whenever
$\gamma>\gamma'$. Because the form $a_\mu$ is symmetric, continuous and strictly
coercive, the form domain of $\tilde{\mathcal{A}}$ is
\begin{equation}\label{eq:form_domain}
    D(\tilde{\mathcal{A}}^{1/2})=Y_{1/2}=V=H^{1}(\Omega_\mu),
    \qquad
    a_\mu(u,u)=\big\|\tilde{\mathcal{A}}^{1/2}u\big\|^{2}_{L^{2}(\Omega_\mu)} .
\end{equation}

The geometry of the domain enters through the compactness of the resolvent. By
\eqref{eq:form_domain}, the embedding $Y_{1/2}\hookrightarrow Y_{0}$ is the
embedding $H^{1}(\Omega_\mu)\hookrightarrow L^{2}(\Omega_\mu)$. On an arbitrary
bounded open set this embedding need not be compact; it is compact here because
$\Omega_\mu$ is a bounded Sobolev extension domain, so that one may extend a
bounded subset of $H^{1}(\Omega_\mu)$ to a bounded subset of
$H^{1}(\mathbb{R}^{n})$ supported in a fixed ball by Theorem \ref{thm:extension}
and a cut-off, apply the Rellich--Kondrachov theorem there, and restrict back.
Consequently $\tilde{\mathcal{A}}^{-1/2}:Y_{0}\to Y_{0}$ is compact, so
$\tilde{\mathcal{A}}$ has compact resolvent: its spectrum consists of eigenvalues
$0<\lambda_{1}\le\lambda_{2}\le\cdots\to\infty$, with an associated orthonormal
basis $\{e_{j}\}_{j\ge1}$ of $L^{2}(\Omega_\mu)$, and
\begin{equation}\label{eq:spectral_norm}
    u=\sum_{j\ge1}c_{j}e_{j}\ \in Y_\gamma
    \iff
    \|u\|^{2}_{Y_\gamma}=\sum_{j\ge1}\lambda_{j}^{2\gamma}|c_{j}|^{2}<\infty ,
    \qquad \gamma\in\mathbb{R} .
\end{equation}
In particular the embeddings $Y_\gamma\hookrightarrow Y_{\gamma'}$ are compact
whenever $\gamma>\gamma'$.

\begin{prop}[Identification of the scales]\label{prop:scale_shift}
Let $\tilde{\mathcal{A}}$ and $\{Y_\gamma\}_{\gamma\in\mathbb{R}}$ be as above.
\begin{enumerate}
    \item[(i)] For every $\beta\in\mathbb{R}$, the operator induced by
    $\tilde{\mathcal{A}}$ on the base space $Y_\beta$ is strictly positive and
    self-adjoint, and its fractional power space of index $\gamma\ge0$ is
    $Y_{\beta+\gamma}$, with equality of norms.

    \item[(ii)] $X_{0}=V'=Y_{-1/2}$ and $X_{1}=V=Y_{1/2}$; consequently
    \begin{equation}\label{eq:scale_identification}
        X_{\alpha+1/2}=Y_\alpha ,\qquad \alpha\ge-1/2 .
    \end{equation}

    \item[(iii)] For every $\beta\ge-1/2$, the fractional power scale of
    $\mathcal{A}$ based at $X_\beta$ is $\gamma\mapsto X_{\beta+\gamma}$.
\end{enumerate}
All identifications hold with equality of norms when $V$ carries the energy norm
$a_\mu(\cdot,\cdot)^{1/2}$, and with $\mu$-independent equivalence of norms when
$V$ carries the usual norm of $H^{1}(\Omega_\mu)$.
\end{prop}

\proof
(i) Fix $\beta\in\mathbb{R}$. By \eqref{eq:spectral_norm}, $Y_\beta$ is a Hilbert
space with inner product
$(u,v)_{Y_\beta}=\sum_{j}\lambda_{j}^{2\beta}c_{j}\overline{d_{j}}$, and the
operator induced by $\tilde{\mathcal{A}}$ on $Y_\beta$ acts by
$u\mapsto\sum_{j}\lambda_{j}c_{j}e_{j}$ on the domain $Y_{\beta+1}$; it is
therefore diagonal in the orthogonal basis
$\{\lambda_{j}^{-\beta}e_{j}\}$ of $Y_\beta$, strictly positive and
self-adjoint. Its fractional powers act by
$u\mapsto\sum_{j}\lambda_{j}^{\gamma}c_{j}e_{j}$, whence
\begin{equation*}
    \Big\|\sum_{j}\lambda_{j}^{\gamma}c_{j}e_{j}\Big\|^{2}_{Y_\beta}
    =\sum_{j}\lambda_{j}^{2\beta+2\gamma}|c_{j}|^{2}
    =\|u\|^{2}_{Y_{\beta+\gamma}} ,
\end{equation*}
which is the assertion.

(ii) By \eqref{eq:form_domain}, $Y_{1/2}=V$. Identifying $Y_{0}$ with its own
dual, the duality induced by the $L^{2}$ pivot gives $V'=(Y_{1/2})'=Y_{-1/2}$,
and therefore $X_{0}=V'=Y_{-1/2}$. Moreover $\mathcal{A}\in\mathcal{L}(V,V')$ is
exactly the extension of $\tilde{\mathcal{A}}$ to this larger space, since
$\langle\mathcal{A}u,v\rangle_{V',V}=a_\mu(u,v)
=(\tilde{\mathcal{A}}^{1/2}u,\tilde{\mathcal{A}}^{1/2}v)_{L^{2}}$ for
$u,v\in V$; in particular $D(\mathcal{A})=V=Y_{1/2}=X_{1}$. Now
\eqref{eq:scale_identification} is (i) with $\beta=-1/2$.

(iii) By (ii), $X_\beta=Y_{\beta-1/2}$, so this is again (i), applied with
$Y_{\beta-1/2}$ as base space.

We note that (i) is nothing but the classical statement that, for a strictly
positive self-adjoint operator, the fractional power spaces coincide with the
complex interpolation spaces of its scale,
\begin{equation}\label{eq:reiteration}
    [\,Y_{a},\,Y_{b}\,]_{\vartheta}=Y_{(1-\vartheta)a+\vartheta b},
    \qquad a\le b,\ \vartheta\in[0,1],
\end{equation}
for which we refer to \cite[Ch.~1, \S6]{LionsMagenes1972} or
\cite[\S1.18.10]{Triebel1978}; for the extrapolated part of the scale see
\cite[Ch.~V]{Amann1995}. The elementary computation above was included because
it makes the origin of the index shift transparent: it is the weight
$\lambda_{j}^{2\beta}$ carried by the base space.
\eproof

Setting $\alpha=0$ in \eqref{eq:scale_identification} we recover
$X_{1/2}=Y_{0}=L^{2}(\Omega_\mu)$, and setting $\alpha=1/2$ we obtain
$X_{1}=Y_{1/2}=H^{1}(\Omega_\mu)$. The compactness of the embeddings of the
$Y$-scale transfers accordingly: for any $\sigma>0$,
\begin{equation}\label{eq:compact_X_scale}
    X_{1+\sigma}=Y_{1/2+\sigma}\hookrightarrow Y_{1/2}=X_{1},
\end{equation}
the embedding being compact. This allows us to work with the phase space
$X_{1}=H^{1}(\Omega_\mu)$ and to use the slightly smaller spaces $X_{1+\sigma}$
to extract compactness for our orbits, completely bypassing the need for
classical spatial regularity beyond $H^{1}$.

\begin{rem}[Scope of the two identifications]\label{rem:scope_identifications}
Two identifications are used in this paper, and they are of different natures.

Proposition \ref{prop:scale_shift} concerns the fractional power scale of a
single self-adjoint operator. It is exact, it holds for every index, and it uses
no property of $\Omega_\mu$ at all. Self-adjointness, on the other hand, is
essential: for a general sectorial operator on a Banach space the fractional
power spaces need not coincide with the complex interpolation spaces, and
additional hypotheses --- such as bounded imaginary powers --- are required.

The identification of these spaces with classical Sobolev spaces is a different
matter. Combining \eqref{eq:form_domain} with Lemma
\ref{lem:uniform_interpolation} and Proposition \ref{prop:scale_shift} we obtain
\begin{equation}\label{eq:sobolev_identification}
    Y_{s/2}=X_{(s+1)/2}=H^{s}(\Omega_\mu),
    \qquad 0\le s\le 1,
\end{equation}
with constants independent of $\mu$; this is where the uniform Jones condition
enters, through the extension operator. The restriction $s\le1$ is not an
artefact of the proof. Complex interpolation transports information supplied at
the endpoints of a couple; it does not create regularity. To reach $s>1$ one
would have to interpolate between $H^{1}(\Omega_\mu)$ and
$D(\tilde{\mathcal{A}})=\{u\in H^{1}(\Omega_\mu):\mathcal{A}u\in
L^{2}(\Omega_\mu)\}$, and on a general Jones domain, with merely bounded
measurable coefficients, this endpoint is not known to be contained in
$H^{2}(\Omega_\mu)$. Accordingly, no inclusion of the form
$X_{1+\sigma}\hookrightarrow H^{1+\tau}(\Omega_\mu)$, $\sigma,\tau>0$, is
available in the present setting, and none is used in what follows: the
compactness supplied by \eqref{eq:compact_X_scale} is a statement about the
abstract scale only.
\end{rem}

To prove the well-posedness of our problem, we appeal to the abstract sectorial
evolutionary framework developed by Henry \cite{Henry1981}. This theory requires
a sectorial operator defined on a base Banach space $\mathcal{X}_{0}$ with domain
$\mathcal{X}_{1}$, and a nonlinearity that is locally Lipschitz from an
intermediate fractional space $\mathcal{X}_\alpha$ into $\mathcal{X}_{0}$, for
some $0\le\alpha<1$.

Our natural phase space for the initial conditions is
$X_{1}=H^{1}(\Omega_\mu)$. If we were to use $X_{0}=(H^{1}(\Omega_\mu))'$ as the
base space, the initial condition space $X_{1}$ would correspond exactly to
$\alpha=1$, which strictly falls outside Henry's admissible range. To circumvent
this and still work with the phase space $H^{1}(\Omega_\mu)$, we fix a parameter
$0<\delta<\frac{1}{2}$ and consider the shifted base space
$\mathcal{Z}_{0}=X_\delta$. By Proposition \ref{prop:scale_shift}(iii), the
operator $\mathcal{A}$ remains strictly positive, self-adjoint and sectorial on
$\mathcal{Z}_{0}$ --- so that it still generates an analytic semigroup there ---
and its fractional power spaces are
\begin{equation}\label{eq:Z_scale}
    \mathcal{Z}_\gamma=X_{\delta+\gamma},\qquad \gamma\ge0 ;
\end{equation}
in particular $\mathcal{Z}_{1}=X_{1+\delta}$.

In this shifted framework the desired phase space $X_{1}=H^{1}(\Omega_\mu)$ is
recovered as a fractional space of index strictly less than one: by
\eqref{eq:Z_scale} with $\gamma=1-\delta$,
\begin{equation}
    \mathcal{Z}_{1-\delta}=X_{\delta+(1-\delta)}=X_{1}=H^{1}(\Omega_\mu),
\end{equation}
and since $\delta>0$ we have $1-\delta<1$, which is the crucial abstract
requirement.

From now on we bypass explicit reference to the shifted scale
$\mathcal{Z}_\gamma$ and work directly within the original $X_\alpha$ scale; by
\eqref{eq:Z_scale} the relevant range is $\delta\le\alpha<1+\delta$, and initial
data will be taken in $X_\alpha$ with $1\le\alpha<1+\delta$. More precisely, we
consider the abstract Cauchy problem
\begin{equation} \label{abstract_problem_scale}
    \begin{cases}
    \dfrac{du}{dt} + \mathcal{A} u = F(u), \quad t \in [0, \tau], \\[2pt]
    u(0) = u_{0} \in X_{\alpha}, \quad 1 \leq \alpha < 1+ \delta,
    \end{cases}
\end{equation}
where $0<\delta<\frac{1}{2}$, $\mathcal{A}$ is the operator defined in
\eqref{def:lin_operator} (acting in the base space $X_\delta$) and
$F(u)(x)=\tilde{f}(u(x))$ is the Nemytskii operator associated with the shifted
nonlinearity defined in Section \ref{sec:preliminaries}.

Initial data will be taken in $X_{\alpha}$ with $1\le\alpha<1+\delta$. The
upper bound is forced: since $F$ takes values in $X_{\delta}$, the variation of
constants integral requires
$\|e^{-\mathcal{A}\tau}\|_{\mathcal{L}(X_{\delta},X_{\alpha})}\le
C\tau^{-(\alpha-\delta)}$ to be integrable, that is $\alpha-\delta<1$. The
lower bound is a choice, and the trade-off it embodies is worth stating.
Nothing in the abstract theory forbids $\alpha<1$; what it costs is a stricter
growth condition, since $F$ must then be locally Lipschitz from $X_{\alpha}$
into $X_{\delta}$, and the admissible exponent $p$ decreases with $\alpha$. The
value $\alpha=1$, for which $X_{\alpha}=H^{1}(\Omega_\mu)$, is exactly the one
that admits the full subcritical range $p<\frac{n+2}{n-2}$; since preserving
this optimal exponent uniformly across the family is one of the points of this
paper, we keep the data in $H^{1}(\Omega_\mu)$, which is in any case the phase
space in which the dynamics is studied throughout.

\begin{rem}[The Two-Dimensional Case]
Throughout this work, the hypotheses on the nonlinearity, the  explicit calculations for the Sobolev embeddings and the Moser-Alikakos iteration are presented for dimensions $n \ge 3$, where one must carefully track the critical Sobolev exponent $2^* = \frac{2n}{n-2}$. The two-dimensional case ($n=2$) is actually simpler, as the uniform continuous embedding $H^1(\Omega_\mu) \hookrightarrow L^q(\Omega_\mu)$ holds for any finite $q \in [2, \infty)$. Consequently, the bootstrap arguments can be directly adapted by choosing an arbitrarily large initial exponent $q$, leading to the same uniform $L^\infty$ bounds. For the sake of conciseness and to avoid repetitive notational distinctions, we carry out the detailed proofs  only for $n \ge 3$.
\end{rem}

\section{Local Existence in a Scale of Banach Spaces} \label{sec:local_existence}

A solution of \eqref{abstract_problem_scale} on $[0,\tau)$ is a continuous function $u: [0, \tau) \to X_{\alpha}$ such that $u(0) = u_0$, the mapping $t \mapsto F(u(t))$ is continuous into $X_\delta$, $u(t) \in D(\mathcal{A})$ for $t>0$, and $u$ satisfies \eqref{abstract_problem_scale} pointwise on $(0, \tau)$. This formulation, due to Miklav\v{c}i\v{c} \cite{Miklavicic1985}, resolves a classic non-uniqueness caveat present in the original definition by Henry \cite{Henry1981} without altering the core analytical proofs. 
  
\begin{theorem}[Local Existence]\label{thm:local_existence}
Assume that $\{\Omega_{\mu}\}_{\mu\in\Lambda}$ is a uniform Jones family of domains, the strict ellipticity condition (Hypothesis \ref{hyp:elliptic_condition}) on $A$, and the subcritical growth condition (Hypothesis \ref{hyp:nonlinear_growth}) on $f$ hold. Then, for any sufficiently small $0 < \delta < \frac{1}{2}$ and any initial data $u_0 \in X_{\alpha}$ with $1 \leq \alpha < 1+ \delta$, there exists a unique solution to the abstract Cauchy problem \eqref{abstract_problem_scale} defined on a maximal interval $[0,T_{max})$. If $T_{max} < \infty$, then 
\begin{equation}
    \limsup_{t \nearrow T_{max}} \|u(t)\|_{X_{\alpha}} = \infty.
\end{equation} 
In addition, if $f \in C^r(\mathbb{R})$, then the assignment $(t, u_0) \mapsto u(t, u_0)$ is a $C^r$ function within its domain of definition.
\end{theorem}

\proof   
Since $\alpha \geq 1$, we have the continuous embedding $X_{\alpha} \hookrightarrow X_1 = H^{1}(\Omega_\mu)$. The standard Sobolev embeddings (Lemma \ref{lem:uniform_sobolev}) then guarantee that $X_{\alpha} \hookrightarrow L^{\frac{2n}{n-2}}(\Omega_\mu)$.

To establish that the abstract nonlinearity maps into our chosen shifted base
space, we exploit the duality properties of the fractional scale. By
\eqref{eq:scale_identification} we have $X_{1-\delta}=Y_{1/2-\delta}$ and
$X_{\delta}=Y_{\delta-1/2}$; since $(Y_{a})'=Y_{-a}$ when
$Y_{0}=L^{2}(\Omega_\mu)$ is taken as the pivot space, $X_\delta$ is the
topological dual of $X_{1-\delta}$. Moreover, by \eqref{eq:sobolev_identification}
applied with $s=1-2\delta$,
\begin{equation}\label{eq:X_1md_identification}
    X_{1-\delta}=H^{1-2\delta}(\Omega_\mu),
\end{equation}
with equivalence constants independent of $\mu$. Corollary
\ref{cor:uniform_sobolev_frac}, again with $s=1-2\delta$, therefore yields the
continuous embedding, uniformly in $\mu$,
\begin{equation}
    X_{1-\delta} \hookrightarrow L^{\frac{2n}{n-2+4\delta}}(\Omega_\mu).
\end{equation}

 By a standard duality argument, this implies that the dual of the Lebesgue space embeds continuously into $(X_{1-\delta})' = X_\delta$. That is, we obtain:
\begin{equation}\label{eq:duality_embedding}
    L^{\frac{2n}{n+2-4\delta}}(\Omega_\mu) \hookrightarrow X_\delta.
\end{equation}

Under the subcritical growth condition (Hypothesis \ref{hyp:nonlinear_growth}), the real function $f$ satisfies $|f(x) - f(y)| \leq C|x - y|(1 + |x|^{p-1} + |y|^{p-1})$. This specific structural bound ensures not only that the Nemytskii operator maps $L^{\frac{2n}{n-2}}(\Omega_\mu)$ into $L^{\frac{2n}{p(n-2)}}(\Omega_\mu)$, but also that it is locally Lipschitz continuous. Indeed, for any $u, v \in L^{\frac{2n}{n-2}}(\Omega_\mu)$, applying H\"older's inequality with conjugate exponents $p$ and $\frac{p}{p-1}$ yields:
\begin{equation}
    \|F(u) - F(v)\|_{L^{\frac{2n}{p(n-2)}}} \leq C \|u - v\|_{L^{\frac{2n}{n-2}}} \left( 1 + \|u\|_{L^{\frac{2n}{n-2}}}^{p-1} + \|v\|_{L^{\frac{2n}{n-2}}}^{p-1} \right).
\end{equation}
Since $f$ is strictly subcritical ($p < \frac{n+2}{n-2}$), we can always choose $\delta > 0$ sufficiently small such that
\begin{equation}
    p \leq \frac{n+2-4\delta}{n-2},
\end{equation}
which implies the continuous inclusion $L^{\frac{2n}{p(n-2)}}(\Omega_\mu) \hookrightarrow L^{\frac{2n}{n+2-4\delta}}(\Omega_\mu)$. Combining this inclusion with \eqref{eq:duality_embedding}, it follows that the Nemytskii operator
\begin{equation}
    F: X_{\alpha} \to X_\delta
\end{equation}
is well-defined and locally Lipschitz continuous on bounded subsets. The conclusion then follows directly from classical results in \cite{Henry1981}, specifically Theorems 3.3.3, 3.3.4, 3.4.4 and Corollaries 3.4.5 and 3.4.6.
\eproof

\section{Global Existence} \label{section:global_existence}

To establish that the local solutions obtained in Theorem \ref{thm:local_existence} can be extended globally in time, we show that their norms cannot blow up in finite time within the natural phase space $X_1 = H^1(\Omega_{\mu})$. Since the local solutions belong to $D(\mathcal{A}) = X_{1+\delta}$ for $t > 0$, proving global well-posedness for any initial data in $X_\alpha$ ($1 \leq \alpha < 1+\delta$) reduces to ruling out the possibility of a finite-time blow-up of the $X_1$-norm. 

We first establish a priori estimates for the solution in Lebesgue spaces and, consequently, for the nonlinear term in the shifted fractional base space $X_\delta$.

\begin{prop}[A Priori Bounds]\label{prop:apriori_bounds}
Assume the hypotheses of Theorem \ref{thm:local_existence} and the dissipativity condition (Hypothesis \ref{hyp:nonlinear_dissipativity}). Let $u(t)$ be a local solution on its maximal interval of existence $[0, T_{max})$. Then, for any $t \in [0, T_{max})$, the norm $\|u(t)\|_{L^q(\Omega_{\mu})}$ (where $q = \frac{2n}{n-2}$) and the non-linearity norm $\|F(u(t))\|_{X_\delta}$ remain bounded on bounded time intervals.
\end{prop}

\proof
From the uniform Sobolev embeddings (Lemma \ref{lem:uniform_sobolev}), we have $X_1 \hookrightarrow L^q(\Omega_{\mu})$ where $q = 2^{*} = \frac{2n}{n-2}$. To establish an a priori bound for the solutions in the $L^q(\Omega_{\mu})$ norm, we exploit the dissipativity of the nonlinearity. Since classical integration by parts is problematic on the non-smooth Jones domain $\Omega_\mu$, we evaluate the governing equation in its weak variational form. Testing the equation with $v = |u|^{q-2}u$, the abstract bilinear form associated with the strictly elliptic operator $\mathcal{A}$ yields the energy identity:
\begin{equation}\label{eq:Lq_energy_identity}
    \frac{1}{q} \frac{d}{dt} \|u\|_{L^q}^q + \frac{4(q-1)}{q^2} \int_{\Omega_{\mu}} A(x) \nabla \left(|u|^{q/2}\right) \cdot \nabla \left(|u|^{q/2}\right) \, dx = \int_{\Omega_{\mu}} f(u)|u|^{q-2}u \, dx.
\end{equation}

By the strict ellipticity condition (Hypothesis \ref{hyp:elliptic_condition}), the diffusion matrix $A$ is positive-definite, so the gradient term on the left-hand side of \eqref{eq:Lq_energy_identity} is non-negative. On the other hand, the dissipativity condition (Hypothesis \ref{hyp:nonlinear_dissipativity}) guarantees the existence of constants $C_1' > 0$ and $C_2' \ge 0$ such that $f(s)|s|^{q-2}s \le -C_1' |s|^q + C_2'$ for all $s \in \mathbb{R}$. Substituting this upper bound into \eqref{eq:Lq_energy_identity} and integrating over $\Omega_{\mu}$, we obtain the differential inequality:
\begin{equation}
    \frac{d}{dt} \|u\|_{L^q}^q \leq -q C_1' \|u\|_{L^q}^q + q C_2' |\Omega_{\mu}|.
\end{equation}

An application of Gronwall's inequality immediately yields the uniform estimate:
\begin{equation}\label{eq:Gronwall_Lq}
    \|u(t)\|_{L^q(\Omega_{\mu})}^q \le e^{-C_1' q t} \|u(0)\|_{L^q(\Omega_{\mu})}^q + \frac{C_2' |\Omega_{\mu}|}{C_1'} \left(1 - e^{-C_1' q t}\right).
\end{equation}

Next, we show that this uniform control in $L^q(\Omega_{\mu})$ implies a control for the nonlinearity in the shifted base space $X_\delta$. Recall from our duality analysis in the proof of Theorem \ref{thm:local_existence} that the continuous embedding
\begin{equation}
    L^{\frac{2n}{n+2-4\delta}}(\Omega_{\mu}) \hookrightarrow X_{\delta}
\end{equation}
holds for any sufficiently small $\delta > 0$. Under the polynomial growth bound inherited from Hypothesis \ref{hyp:nonlinear_growth}, the Nemytskii operator $F(u)$ maps $L^q(\Omega_{\mu})$ continuously into $L^{q/p}(\Omega_{\mu})$. Since $f$ is strictly subcritical ($p < \frac{n+2}{n-2}$), we can fix $\delta > 0$ small enough to satisfy the algebraic condition $p \leq \frac{n+2-4\delta}{n-2}$, which ensures the continuous inclusion $L^{q/p}(\Omega_{\mu}) \hookrightarrow L^{\frac{2n}{n+2-4\delta}}(\Omega_{\mu})$.

Consequently, the uniform $L^q(\Omega_{\mu})$ bound established in \eqref{eq:Gronwall_Lq} guarantees that the nonlinear term $F(u(t))$ remains uniformly bounded in $X_\delta$ on any bounded time interval. This intermediate control in the lower-regularity space is the critical ingredient required to prevent the blow-up of the phase-space norm of the solution.
\eproof

\begin{rem} \label{rem:truncation}
The identity \eqref{eq:Lq_energy_identity} is obtained, initially, only
formally: the `test function' $|u|^{q-2}u$ need not belong to
$H^{1}(\Omega_\mu)$. It is made rigorous by truncating $u$ rather than the test
function, that is, by testing with $u\,u_M^{\,q-2}$ where $u_M=\min\{|u|,M\}$
--- a genuine element of $H^{1}(\Omega_\mu)$, by \eqref{eq:grad_phiM}---
and letting $M\to\infty$ by monotone convergence. This is carried out in detail,
for an arbitrary exponent $p\ge q$, in Steps 1, 2 and 4 of the proof of Lemma
\ref{lem:Lp_energy}; the case $p=q$ needed here is contained in it, and we postpone the detailed argument  to avoid repetition.
\end{rem}
\begin{theorem}[Global Existence]\label{thm:global_existence}
Assume that $\{\Omega_{\mu}\}_{\mu\in\Lambda}$ is a uniform Jones family of domains and that the strict ellipticity condition (Hypothesis \ref{hyp:elliptic_condition}), the subcritical growth condition (Hypothesis \ref{hyp:nonlinear_growth}), and the dissipativity condition (Hypothesis \ref{hyp:nonlinear_dissipativity}) hold. Then, for any initial data $u_0 \in X_1$, the maximal time of existence is $T_{max} = \infty$.
\end{theorem}

\proof
Assume, for the sake of contradiction, that $T_{max} < \infty$. According to the blow-up alternative established in Theorem \ref{thm:local_existence}, this implies that $\limsup_{t \nearrow T_{max}} \|u(t)\|_{X_1} = \infty$.

Using the variation of constants formula in the shifted fractional scale with base space $X_\delta$, we express the solution as:
\begin{equation}\label{eq:var_constants_Xdelta}
    u(t) = e^{-\mathcal{A} t} u_0 + \int_0^t e^{-\mathcal{A}(t-s)} F(u(s)) \, ds.
\end{equation}

Taking the $X_1$ norm on both sides of \eqref{eq:var_constants_Xdelta}, we estimate the linear and integral terms separately. Since $\mathcal{A}$ generates an analytic semigroup, it is strongly continuous on $X_1$, yielding $\|e^{-\mathcal{A} t} u_0\|_{X_1} \leq \tilde{M} \|u_0\|_{X_1}$ for some constant $\tilde{M} > 0$ defined over the finite interval $[0, T_{max})$.

For the integral term, we exploit the standard smoothing estimates of analytic semigroups acting from the base space $X_\delta$ into the fractional space $X_1$. Since the difference between these fractional indices is exactly $1 - \delta$, there exists a constant $M_\delta > 0$ such that:
\begin{align}
    \int_0^t \|e^{-\mathcal{A}(t-s)} F(u(s))\|_{X_1} \, ds 
    &\leq \int_0^t M_\delta (t-s)^{-(1-\delta)} \|F(u(s))\|_{X_\delta} \, ds \nonumber \\
    &\leq M_\delta K \int_0^t (t-s)^{-(1-\delta)} \, ds,
\end{align}
where $K = \sup_{s \in [0, T_{max})} \|F(u(s))\|_{X_\delta} < \infty$, courtesy of Proposition \ref{prop:apriori_bounds}. Because $\delta > 0$, the exponent $1-\delta$ is strictly less than $1$, ensuring that the weakly singular integral converges. Explicit evaluation yields the a priori estimate:
\begin{equation}\label{eq:X1_uniform_bound}
    \|u(t)\|_{X_1} \leq \tilde{M} \|u_0\|_{X_1} + \frac{M_\delta K}{\delta} t^\delta.
\end{equation}

Taking the limit as $t \nearrow T_{max} < \infty$, the right-hand side of \eqref{eq:X1_uniform_bound} remains strictly bounded. This shows that $\|u(t)\|_{X_1}$ is bounded on $[0, T_{max})$, which directly contradicts the blow-up alternative. Thus, we must have $T_{max} = \infty$.
\eproof

\begin{rem}\label{rem:non_gradient_structure}
Although global existence for parabolic equations of this type is frequently established using a Lyapunov functional---relying heavily on the gradient structure of the system to obtain a priori energy bounds---we have deliberately avoided this approach. The proof of Theorem \ref{thm:global_existence} relies exclusively on the dissipativity condition in $L^q(\Omega_\mu)$ and the regularising properties of the analytic semigroup in fractional power spaces. Consequently, this methodology does not require the underlying system to possess a gradient structure, rendering the global existence result applicable to a strictly broader class of dissipative non-gradient systems, including those subject to non-variational perturbations or external forcing.
\end{rem}

\section{Existence of the Global Attractor}
\label{sec:global_attractor_section}

\subsection{Existence of a Bounded Absorbing Ball}

Our first step in analyzing the long-time behavior of the system is to establish the existence of an absorbing set in the Lebesgue space ${L^q}({\Omega}_{\mu})$, where $q = \frac{2n}{n-2}$ is the critical Sobolev exponent. We denote by $|\Omega|$ the Lebesgue measure of the domain $\Omega$ and say that the family $\{\Omega_{\mu}\}_{\mu \in\Lambda}$ has uniformly bounded volume if $\sup_{\mu \in\Lambda} |\Omega_{\mu}| \leq V_0$, for some constant $V_0$; recall that Definition \ref{def:uniform_Jones} also
requires $\operatorname{diam}(\Omega_\mu)\ge d_{*}>0$, which, as noted in
Remark \ref{rem:scale_normalisation}, holds in particular whenever the volumes
are bounded away from zero.

\begin{prop}[Absorbing Set in $L^q$] \label{prop:absorbing_Lq} 
Suppose that $\{\Omega_{\mu}\}_{\mu \in \Lambda}$ is a uniform Jones family of domains with uniformly bounded volume. Also, suppose the strict ellipticity condition (Hypothesis \ref{hyp:elliptic_condition}) on the matrix $A$, the subcritical growth condition (Hypothesis \ref{hyp:nonlinear_growth}), and the dissipativity condition (Hypothesis \ref{hyp:nonlinear_dissipativity}) hold. Let $B \subset X_1 = H^1(\Omega_\mu)$ be an arbitrary bounded set, such that $\|u_0\|_{X_1} \leq M$ for all $u_0 \in B$. Then, there exists a constant $\rho_{L^q} > 0$, independent of the initial data and the parameter $\mu$, and a time $t_0 = t_0(M) \geq 0$ such that the global solution $u(t)$ satisfies:
\begin{equation} \label{eq:absorbing_Lq_final}
    \|u(t)\|_{L^q(\Omega_\mu)} \leq \rho_{L^q}, \quad \forall t \geq t_0.
\end{equation}
\end{prop}

\proof
By the uniform Sobolev embeddings (Lemma \ref{lem:uniform_sobolev}), the embedding $X_1 \hookrightarrow L^q(\Omega_\mu)$ is continuous with a uniform constant $C_S > 0$. Consequently, the initial data satisfies the uniform bound:
\begin{equation} \label{eq:initial_lq_bound}
    \|u_0\|_{L^q} \leq C_S \|u_0\|_{X_1} \leq C_S M, \quad \forall u_0 \in B.
\end{equation}

Recall that in Proposition \ref{prop:apriori_bounds}, by testing the equation in its weak formulation, we established the explicit a priori estimate \eqref{eq:Gronwall_Lq}:
\begin{equation*} 
    \|u(t)\|_{L^q(\Omega_{\mu})}^q \le e^{-C_1' q t} \|u_0\|_{L^q(\Omega_{\mu})}^q + \frac{C_2' |\Omega_{\mu}|}{C_1'} \left(1 - e^{-C_1' q t}\right),
\end{equation*}
where $C_1' > 0$ and $C_2' \geq 0$ are the structural constants from the dissipativity condition.

Using the uniform volume bound $|\Omega_\mu| \leq V_0$ and the initial data bound \eqref{eq:initial_lq_bound}, this inequality yields:
\begin{equation} \label{eq:Gronwall_solved_Lq}
    \|u(t)\|_{L^q}^q \leq (C_S M)^q e^{-C_1' q t} + \frac{C_2' V_0}{C_1'}.
\end{equation}

Since $C_2'\ge0$ may always be replaced by a larger constant without affecting
\eqref{eq:Gronwall_Lq}, we may and do assume $C_2'>0$. We now define the uniform
radius of the absorbing ball in $L^q(\Omega_\mu)$ as:
\begin{equation} \label{eq:radius_Lq}
    \rho_{L^q} = \left( \frac{2 C_2' V_0}{C_1'} \right)^{\frac{1}{q}}.
\end{equation}
Crucially, the radius $\rho_{L^q}$ depends solely on the uniform maximum volume $V_0$ and the structural constants of the equation, being completely independent of the initial data and of the perturbation parameter $\mu$. 

For the given bounded set $B \subset X_1$, we define the uniform entering time $t_0 = t_0(M) \geq 0$ by determining when the transient term becomes smaller than half of the asymptotic bound:
\begin{equation} \label{eq:entering_time}
    t_0 = \max \left\{ 0, \, \frac{1}{q C_1'} \ln \left( \frac{C_1' (C_S M)^q}{C_2' V_0} \right) \right\}.
\end{equation}
It follows directly from \eqref{eq:Gronwall_solved_Lq} and \eqref{eq:entering_time} that for all $t \geq t_0$, $\|u(t)\|_{L^q} \leq \rho_{L^q}$. This completes the proof.
\eproof

\subsection{Regularisation in the Extended Fractional Scale and Global Attractor}

We now use the uniform $L^q$-bound and the regularising properties of the flow to construct an absorbing set in the phase space $X_1 = H^1(\Omega_\mu)$. To utilise the fractional powers rigorously, we return to the abstract formulation $u_t + \mathcal{A}u = F(u)$, where $\mathcal{A}$ is the strictly positive shifted operator.

\begin{prop}[Absorbing Ball in $X_1$] \label{prop:absorbing_X1}
Under the assumptions of Proposition \ref{prop:absorbing_Lq}, there exists a constant $\rho_{H^1} > 0$, independent of the initial data and the parameter $\mu$, such that the global solution $u(t)$ satisfies:
\begin{equation} \label{eq:absorbing_H1_final}
    \|u(t)\|_{X_1} \leq \rho_{H^1}, \quad \forall t \geq t_0 + 1,
\end{equation}
where $t_0$ is the entering time defined in \eqref{eq:entering_time}. Consequently, the dynamical system admits a bounded absorbing ball $\mathcal{B}_{H^1} \subset X_1$, with a radius independent of the parameter $\mu$.
\end{prop}

\proof
Due to the subcritical growth condition assumed on the nonlinearity $f$, we demonstrated in Proposition \ref{prop:apriori_bounds} that $F(u)$ maps $L^q(\Omega_\mu)$ continuously into the shifted base space $X_\delta$, for a sufficiently small $\delta > 0$. The uniform bound \eqref{eq:absorbing_Lq_final} thus implies that the Nemytskii operator maps the $L^q$-absorbing ball into a uniformly bounded set of $X_\delta$. Specifically, there exists a constant $C_F = C_F(\rho_{L^q}) > 0$, independent of $\mu$, such that:
\begin{equation} \label{eq:source_bounded}
    \|F(u(t))\|_{X_\delta} \leq C_F, \quad \forall t \geq t_0.
\end{equation}

For any $t \geq t_0 + 1$, we invoke the variation of constants formula starting from the point $t-1$:
\begin{equation} \label{eq:duhamel}
    u(t) = e^{-\mathcal{A}} u(t-1) + \int_{t-1}^t e^{-\mathcal{A}(t-s)} F(u(s)) \,ds.
\end{equation}

We compute the norm of \eqref{eq:duhamel} in the phase space $X_1$. For the linear term, since $u(t-1)$ belongs to the $L^q$ absorbing ball and $q = \frac{2n}{n-2} > 2$, we apply H\"older's inequality and the uniform volume bound $|\Omega_\mu| \leq V_0$ to bound it in the pivot space $L^2(\Omega_\mu) \equiv X_{1/2}$:
\begin{equation*}
    \|u(t-1)\|_{X_{1/2}} = \|u(t-1)\|_{L^2} \leq |\Omega_\mu|^{\frac{q-2}{2q}} \|u(t-1)\|_{L^q} \leq V_0^{\frac{q-2}{2q}} \rho_{L^q}.
\end{equation*}
Applying the smoothing property of the analytic semigroup from $X_{1/2}$ to $X_1$, with uniform constant $M_{1/2}$, we have:
\begin{equation*}
    \|e^{-\mathcal{A}}u(t-1)\|_{X_1} \leq M_{1/2} (1)^{-1/2} \|u(t-1)\|_{X_{1/2}} \leq M_{1/2} V_0^{\frac{q-2}{2q}} \rho_{L^q} =: C_\ast \rho_{L^q}.
\end{equation*}

For the integral term, since the source is uniformly bounded in $X_\delta$ and the fractional difference is $1-\delta < 1$, the singularity at $s=t$ is integrable:
\begin{align*}
    \int_{t-1}^t \|e^{-\mathcal{A}(t-s)} F(u(s))\|_{X_1} \,ds 
    &\leq \int_{t-1}^t M_\delta (t-s)^{-(1-\delta)} \|F(u(s))\|_{X_\delta} \,ds \\
    &\leq M_\delta C_F \int_{t-1}^t (t-s)^{-(1-\delta)} \,ds = \frac{M_\delta C_F}{\delta}.
\end{align*}

Combining both estimates, we arrive at a uniform bound in the phase space for all large times:
\begin{equation} \label{eq:absorbing_H1}
    \|u(t)\|_{X_1} \leq C_\ast \rho_{L^q} + \frac{M_\delta C_F}{\delta} =: \rho_{H^1}, \quad \forall t \geq t_0 + 1.
\end{equation}
Since all constants involved are independent of $\mu$, the radius $\rho_{H^1}$ is uniform, guaranteeing the existence of the uniform bounded absorbing ball $\mathcal{B}_{H^1} \subset X_1$.
\eproof

We summarise the long-time dynamical outcome in the following main result.

\begin{theorem}[Existence of the Global Attractor]\label{thm:global_attractor} 
Assume that $\{\Omega_{\mu}\}_{\mu \in \Lambda}$ is a uniform Jones family of domains with uniformly bounded volume, and that the strict ellipticity condition (Hypothesis \ref{hyp:elliptic_condition}), the subcritical growth condition (Hypothesis \ref{hyp:nonlinear_growth}), and the dissipativity condition (Hypothesis \ref{hyp:nonlinear_dissipativity}) hold. Then the dynamical system $\{S_\mu(t)\}_{t \geq 0}$ generated by the problem on $X_1 = H^1(\Omega_\mu)$ possesses a unique global attractor $\mathscr{A}_\mu$. The family of attractors $\{\mathscr{A}_\mu\}_{\mu \in \Lambda}$ is uniformly bounded in $H^1(\Omega_{\mu})$. 
\end{theorem}

\proof
To conclude the existence of the global attractor, we shift the Duhamel integration interval slightly further to gain higher regularity, exploiting the asymptotic compactness of the system. By estimating $u(t)$ in $X_{1+\sigma}$ for a sufficiently small $\sigma > 0$ (such that $1+\sigma - \delta < 1$), an identical smoothing argument applied on the interval $[t-1, t]$ starting from $t \geq t_0 + 2$ proves that the trajectory enters a uniformly bounded set of $X_{1+\sigma}$. 

Since the fractional power space $X_{1+\sigma}$ is compactly embedded into $X_1$, the nonlinear semigroup $\{S_\mu(t)\}_{t \geq 0}$ is completely continuous for $t > 0$. Therefore, by the standard theory of infinite-dimensional dissipative dynamical systems, the existence of the absorbing set $\mathcal{B}_{H^1}$ combined with the asymptotic compactness ensures that the system possesses a unique, compact global attractor
$\mathscr{A}_\mu \subset X_1$, which is moreover connected, since $X_1$ is a
connected metric space and each $S_\mu(t)$ is continuous; it is defined as the $\omega$-limit set of the absorbing ball:
\begin{equation*}
    \mathscr{A}_\mu = \omega(\mathcal{B}_{H^1}) = \bigcap_{\tau \geq 0} \overline{\bigcup_{t \geq \tau} S_\mu(t)\mathcal{B}_{H^1}}^{X_1}.
\end{equation*}
Since the absorbing balls are uniformly bounded in $H^1(\Omega_{\mu})$ by Proposition \ref{prop:absorbing_X1}, the result follows.
\eproof

 \section{Uniform Boundedness of the Attractors}
\label{sec:uniform_linfty}

Having established the existence of the global attractors $\mathscr{A}_\mu$ in $X_1$ and a uniformly bounded absorbing ball in $L^q(\Omega_\mu)$ (where $q=\frac{2n}{n-2}$), we now show that bounded sets in $L^q(\Omega_\mu)$ are absorbed into a uniformly bounded set in $L^\infty(\Omega_\mu)$ in finite time. 

\textbf{Standing Assumptions for Section \ref{sec:uniform_linfty}:} To avoid repetition, throughout this section we assume that the family of domains $\{\Omega_\mu\}_{\mu \in \Lambda}$ is a uniform Jones family with uniformly bounded volume, satisfying:
\begin{equation}\label{eq:uniform_measure_bound2}
    \sup_{\mu \in \Lambda} |\Omega_\mu| \leq V_0 < \infty.
\end{equation}
Furthermore, we assume that the strict ellipticity condition (Hypothesis \ref{hyp:elliptic_condition}), the subcritical growth condition (Hypothesis \ref{hyp:nonlinear_growth}), and the dissipativity condition (Hypothesis \ref{hyp:nonlinear_dissipativity}) all hold.
We also fix a bounded set $B\subset X_1=H^1(\Omega_\mu)$, say
$\|u_0\|_{X_1}\le M$ for every $u_0\in B$, and whenever we refer to $u(t)$ we
mean the global weak solution to the initial-boundary value problem associated
with the perturbed equation, generated by an initial datum $u_0\in B$.

By Proposition \ref{prop:absorbing_Lq} both the radius \eqref{eq:radius_Lq} and
the entering time \eqref{eq:entering_time} are independent of $\mu$: the
exponential rate $C_1'$ in \eqref{eq:Gronwall_solved_Lq} comes from the
dissipativity condition alone, and the remaining constants involve only $V_0$,
the uniform Sobolev constant and $M$. Writing $t_0=t_0(M)$ for that entering
time and $C_0:=\rho_{L^q}$, we therefore have:

\begin{equation}\label{eq:uniform_Lq}
    \sup_{t \geq t_0} \|u(t)\|_{L^q(\Omega_\mu)} \leq C_0.
\end{equation}

This absorption is established by a Moser--Alikakos iterative procedure
combined with the Uniform Gronwall Lemma. Throughout, we set

\begin{equation}\label{eq:L_constant}
    L:=\max\Big\{1,\ \sqrt{2C_f/\beta}\,\Big\},
\end{equation}
where $\beta>0$ and $C_f\ge0$ are the constants of the dissipativity condition
(Hypothesis \ref{hyp:nonlinear_dissipativity}); note that $L$ depends only on the
structure of the nonlinearity, and in particular neither on $p$ nor on $\mu$.

Before turning to the iteration we record an elementary bound, which already
settles the boundedness of the attractors themselves and which we shall want to
contrast with what the iteration provides.

\begin{prop}[Elementary bound on the attractors]\label{prop:truncation_linfty}
Under the standing assumptions of this section, with $L$ as in
\eqref{eq:L_constant},
\begin{equation*}
    \sup_{\mu\in\Lambda}\ \sup_{v\in\mathscr{A}_\mu}
    \|v\|_{L^{\infty}(\Omega_\mu)}\ \le\ L .
\end{equation*}
\end{prop}

\proof
Let $u$ be a complete bounded trajectory in $\mathscr{A}_\mu$ and put
$w:=(u-L)^{+}$. Since $\sigma\mapsto(\sigma-L)^{+}$ is Lipschitz and vanishes at
$0$, we have $w(t)\in H^{1}(\Omega_\mu)$ with $\nabla w=\chi_{\{u>L\}}\nabla u$
almost everywhere. Using $w$ as a test function --- which is legitimate because,
the boundary condition being homogeneous Neumann, the test space is all of
$H^{1}(\Omega_\mu)$ --- gives
\begin{equation*}
    \frac{1}{2}\frac{d}{dt}\|w\|^{2}_{L^{2}(\Omega_\mu)}
    +\int_{\Omega_\mu}A(x)\nabla w\cdot\nabla w\,dx
    =\int_{\{u>L\}}f(u)\,(u-L)\,dx .
\end{equation*}
The second term on the left is nonnegative by Hypothesis
\ref{hyp:elliptic_condition}. For $\sigma\ge L$ we have $\sigma^{2}\ge2C_f/\beta$,
hence $f(\sigma)\sigma\le-\beta\sigma^{2}+C_f\le-\tfrac{\beta}{2}\sigma^{2}$ and
therefore $f(\sigma)\le-\tfrac{\beta}{2}\sigma$; since $u>L\ge0$ implies
$u\,(u-L)\ge(u-L)^{2}$, the right-hand side is at most
$-\tfrac{\beta}{2}\|w\|^{2}_{L^{2}(\Omega_\mu)}$. Consequently
\begin{equation*}
    \|w(t)\|^{2}_{L^{2}(\Omega_\mu)}
    \ \le\ e^{-\beta(t-\tau)}\,\|w(\tau)\|^{2}_{L^{2}(\Omega_\mu)},
    \qquad \tau \le t .
\end{equation*}
The trajectory being bounded in $L^{2}(\Omega_\mu)$ uniformly in
$t\in\mathbb{R}$ by Theorem \ref{thm:global_attractor}, letting
$\tau \to-\infty$ yields $w(t)=0$, that is $u\le L$. The same argument applied
to $-u$ gives $u\ge-L$.
\eproof

\begin{prop}[Rate of absorption]\label{cor:absorption_rate}
Under the standing assumptions of this section, for every $\tau$ and every
$t\ge\tau$,
\begin{equation*}
    \big\|(|u(t)|-L)^{+}\big\|_{L^{\infty}(\Omega_\mu)}
    \ \le\ e^{-\beta(t-\tau)/2}\,
    \big\|(|u(\tau)|-L)^{+}\big\|_{L^{\infty}(\Omega_\mu)} .
\end{equation*}
In particular, by Theorem \ref{thm:uniform_linfty},
\begin{equation*}
    \big\|(|u(t)|-L)^{+}\big\|_{L^{\infty}(\Omega_\mu)}
    \ \le\ C_\infty\,e^{-\beta(t-t_{0}-1)/2},
    \qquad t\ge t_{0}+1,
\end{equation*}
so that the trajectories converge to the ball of radius $L$ of
$L^{\infty}(\Omega_\mu)$ at the exponential rate $\beta/2$, uniformly in $\mu$.
\end{prop}

\proof
Let $w:=(u-L)^{+}$ and let $p\ge2$. Testing the weak formulation with
$w^{\,p-1}$, truncated as in the proof of Lemma \ref{lem:Lp_energy} to make it
admissible, gives
\begin{equation*}
    \frac{1}{p}\frac{d}{dt}\|w\|^{p}_{L^{p}(\Omega_\mu)}
    +(p-1)\int_{\Omega_\mu}w^{\,p-2}A(x)\nabla w\cdot\nabla w\,dx
    =\int_{\{u>L\}}f(u)\,w^{\,p-1}\,dx .
\end{equation*}
The diffusion term is nonnegative, and on $\{u>L\}$ one has
$f(u)\le-\tfrac{\beta}{2}u\le-\tfrac{\beta}{2}w$, since $u=w+L\ge w$ there.
Hence $\frac{d}{dt}\|w\|^{p}_{L^{p}}\le-\tfrac{p\beta}{2}\|w\|^{p}_{L^{p}}$,
that is $\|w(t)\|_{L^{p}}\le e^{-\beta(t-\tau)/2}\|w(\tau)\|_{L^{p}}$, with a
rate independent of $p$. Letting $p\to\infty$ gives the first inequality for
$w=(u-L)^{+}$; the same argument applied to $-u$ gives it for $(-u-L)^{+}$, and
$(|u|-L)^{+}$ is the larger of the two.
\eproof
\begin{rem}\label{rem:why_moser}
Proposition \ref{prop:truncation_linfty} is the parabolic counterpart of
\cite[Lemma~B.1~iv)]{ArrietaCarvalhoRodriguezBernal2000}, which asserts that,
for the corresponding elliptic problem and in the presence of a strictly
positive zeroth-order absorption, the $L^{\infty}$ bound is independent of the
domain. It is worth being explicit about what the argument
does \emph{not} use: no extension operator, no Sobolev constant, no bound on
$|\Omega_\mu|$ and no regularity of $\partial\Omega_\mu$ --- only the fact that
the test space of the homogeneous Neumann formulation is $H^{1}(\Omega_\mu)$
itself.
What it does use is the existence of the attractors and their uniform
boundedness in $H^{1}(\Omega_\mu)$, and it is in Sections
\ref{sec:preliminaries}--\ref{sec:global_attractor_section} above that our
geometric assumptions are used to establish this.

It is worth recording why the comparison results of
\cite[Appendix~A]{ArrietaCarvalhoRodriguezBernal2000} are not available here. Their Theorem A.12
compares a solution with a sub- or supersolution lying in
$C((t_{0},t_{1}],X^{1+\epsilon})$ for some $\epsilon>0$, under the hypothesis
that $f(t,\cdot)+\beta I$ be increasing on balls of $X^{1+\epsilon}$, and is
proved by a fixed-point argument carried out in that space, the conclusion
being transferred to $X^{1}$ by density of the positive cones. For a Nemytskii
operator the monotonicity hypothesis amounts to a bound on the range of the
functions in such a ball, that is, to the embedding
$X^{1+\epsilon}\hookrightarrow L^{\infty}$; on a Jones domain with merely
bounded measurable coefficients the fractional power spaces of order above one
are not identified with any space for which such an embedding is known (Remark
\ref{rem:scope_identifications}), so the hypothesis cannot even be formulated.

The order structure required by their Definition A.8, on the other hand, is
available here: the forms $a_\mu$ satisfy the first Beurling--Deny criterion on
any open set, so the semigroups and resolvents are positivity preserving, and
$e^{-\mathcal{A}_\mu t}u\to u$ in $X^{\beta}$ with
$e^{-\mathcal{A}_\mu t}u\in\bigcap_{\alpha}X^{\alpha}$ for $t>0$ gives the
density of the positive cones. The obstruction is the one just described, and
it is the only one.

Proposition
\ref{prop:truncation_linfty} avoids both difficulties by never leaving
$H^{1}(\Omega_\mu)$, and it gives slightly more than a comparison argument
would: the pointwise estimate $f(\sigma)\le-\tfrac{\beta}{2}\sigma$ for
$\sigma\ge L$ yields exponential decay rather than the Gronwall growth that a
local Lipschitz bound on $f$ would produce, and it is this decay that allows
the initial time to be sent to $-\infty$ along a complete bounded trajectory.

The remainder of this section proves what that argument does not give. A
truncation propagates integrability but never creates it: it carries
$L^{\infty}$ data to $L^{\infty}$ data, and the phase space here is
$X_1=H^{1}(\Omega_\mu)$, which for $n\ge2$ is not contained in
$L^{\infty}(\Omega_\mu)$. What is needed, and what Theorem
\ref{thm:uniform_linfty} supplies, is a \emph{smoothing} estimate: bounded sets
of $L^{q}(\Omega_\mu)$ are carried into a fixed ball of $L^{\infty}(\Omega_\mu)$
in finite time, uniformly in $\mu$.

The route of \cite[Lemma~4.2]{ArrietaCarvalhoRodriguezBernal2000} proceeds
instead through the domains of the integer powers of the operator and through
its realisations on the whole Lebesgue scale.
 The iteration below stays within the
$L^{2}$-based scale and uses no realisation of the operator outside it: it is an
energy argument, applying at every step the single inequality
$H^{1}(\Omega_\mu)\hookrightarrow L^{2\chi}(\Omega_\mu)$ to the function
$|u|^{p_k/2}$, the growing exponents being Lebesgue exponents of the iterated
functions rather than Sobolev indices.

\end{rem}

\begin{lemma}[$L^p$ Energy Estimate] \label{lem:Lp_energy}
Under the standing assumptions of this section, let $p\ge q$ and let $u(t)$ be a
global weak solution with $\|u(s)\|_{L^{p}(\Omega_\mu)}<\infty$ for some
$s\ge t_0$. Then $\|u(t)\|_{L^{p}(\Omega_\mu)}<\infty$ for every $t\ge s$, and
\begin{equation}\label{eq:energy_estimate}
    \frac{d}{dt} \int_{\Omega_\mu} |u(t)|^p \, dx
    + 2\alpha \int_{\Omega_\mu} \Big| \nabla \big(|u(t)|^{p/2}\big) \Big|^2 \, dx
    + \frac{p\beta}{2}\int_{\Omega_\mu}|u(t)|^{p}\,dx
    \;\le\; \beta\,V_0\,L^{p}
\end{equation}
in the sense of distributions on $(s,\infty)$, with $L$ as in
\eqref{eq:L_constant}.
\end{lemma}

\proof
Since the spectral shift cancels, we may work directly with the weak
formulation of the original problem: for every $\varphi\in H^{1}(\Omega_\mu)$,
\begin{equation}\label{eq:weak_form_p}
    \langle u_t,\varphi\rangle
    +\int_{\Omega_\mu}A(x)\nabla u\cdot\nabla\varphi\,dx
    =\int_{\Omega_\mu}f(u)\,\varphi\,dx .
\end{equation}
The natural test function for an $L^{p}$ estimate, $u|u|^{p-2}$, need not belong
to $H^{1}(\Omega_\mu)$, since no bound on $u$ in $L^{\infty}$ is available at
this stage. We therefore truncate: for $M>0$ set
\begin{equation*}
    u_M:=\min\{|u|,M\},\qquad
    \varphi_M:=u\,u_M^{\,p-2},\qquad
    w_M:=u_M^{\,p/2} .
\end{equation*}
Since $u_M\le M$, the map $\sigma\mapsto\sigma\min\{|\sigma|,M\}^{p-2}$ is
Lipschitz and vanishes at the origin, so $\varphi_M\in H^{1}(\Omega_\mu)$ and is
an admissible test function, with
\begin{equation}\label{eq:grad_phiM}
    \nabla\varphi_M
    =\big[1+(p-2)\chi_{\{|u|<M\}}\big]\,u_M^{\,p-2}\,\nabla u ,
\end{equation}
while
\begin{equation}\label{eq:grad_wM}
    |\nabla w_M|^{2}
    =\frac{p^{2}}{4}\,u_M^{\,p-2}\,\chi_{\{|u|<M\}}\,|\nabla u|^{2} .
\end{equation}
We test \eqref{eq:weak_form_p} with $p\,\varphi_M$ and estimate the two terms on
the right separately.

For the diffusion term, \eqref{eq:grad_phiM} gives
\begin{equation*}
    p\int_{\Omega_\mu}A(x)\nabla u\cdot\nabla\varphi_M\,dx
    =p\int_{\Omega_\mu}\big[1+(p-2)\chi_{\{|u|<M\}}\big]u_M^{\,p-2}\,
      A(x)\nabla u\cdot\nabla u\,dx ,
\end{equation*}
and the integrand is nonnegative by Hypothesis
\ref{hyp:elliptic_condition}. Discarding the contribution of
$\{|u|\ge M\}$, where the bracket equals $1$, and using strict ellipticity on
the remaining set, where it equals $p-1$, we obtain
\begin{align*}
    p\int_{\Omega_\mu}A(x)\nabla u\cdot\nabla\varphi_M\,dx
    &\;\ge\;\alpha\,p\,(p-1)\int_{\{|u|<M\}}u_M^{\,p-2}|\nabla u|^{2}\,dx\\
    &\;=\;\frac{4\alpha(p-1)}{p}\int_{\Omega_\mu}|\nabla w_M|^{2}\,dx ,
\end{align*}
by \eqref{eq:grad_wM}. Since $4(p-1)\ge 2p$ for $p\ge2$, this yields the bound
in the form we shall use,
\begin{equation}\label{eq:diffusion_lower}
    p\int_{\Omega_\mu}A(x)\nabla u\cdot\nabla\varphi_M\,dx
    \;\ge\;2\alpha\int_{\Omega_\mu}|\nabla w_M|^{2}\,dx .
\end{equation}

For the nonlinear term, $u_M^{\,p-2}\ge0$ and $u^{2}\ge u_M^{2}$, so the
dissipativity condition (Hypothesis \ref{hyp:nonlinear_dissipativity}) gives
pointwise
\begin{equation*}
    f(u)\,\varphi_M=f(u)\,u\,u_M^{\,p-2}
    \;\le\;-\beta\,u^{2}u_M^{\,p-2}+C_f\,u_M^{\,p-2}
    \;\le\;-\beta\,w_M^{2}+C_f\,u_M^{\,p-2} .
\end{equation*}
By H\"older's inequality with exponents $\frac{p}{p-2}$ and $\frac{p}{2}$,
\begin{equation*}
    C_f\int_{\Omega_\mu}u_M^{\,p-2}\,dx
    \;\le\;C_f\,|\Omega_\mu|^{2/p}
    \Big(\int_{\Omega_\mu}w_M^{2}\,dx\Big)^{\frac{p-2}{p}} ,
\end{equation*}
and Young's inequality with the same exponents, with the parameter chosen so
that the absorbed term is $\frac{\beta}{2}\int_{\Omega_\mu} w_M^{2}\,dx$,
gives
\begin{align*}
    C_f\,|\Omega_\mu|^{2/p}
    \Big(\int_{\Omega_\mu}w_M^{2}\,dx\Big)^{\frac{p-2}{p}}
    &\;\le\;\frac{\beta}{2}\int_{\Omega_\mu}w_M^{2}\,dx
    \;+\;\beta\,p^{-p/2}(p-2)^{\frac{p}{2}-1}
      \Big(\frac{2C_f}{\beta}\Big)^{p/2}|\Omega_\mu|\\
    &\;\le\;\frac{\beta}{2}\int_{\Omega_\mu}w_M^{2}\,dx
    \;+\;\frac{\beta}{p}\Big(\frac{2C_f}{\beta}\Big)^{p/2}|\Omega_\mu| .
\end{align*}
Collecting the last three displays and multiplying by $p$,
\begin{equation}\label{eq:nonlinear_upper}
    p\int_{\Omega_\mu}f(u)\,\varphi_M\,dx
    \;\le\;-\frac{p\beta}{2}\int_{\Omega_\mu}w_M^{2}\,dx
    \;+\;\beta\Big(\frac{2C_f}{\beta}\Big)^{p/2}|\Omega_\mu| .
\end{equation}

Finally, set
\begin{equation*}
    G_M(\zeta)=\int_0^{\zeta}\varsigma\,\min\{|\varsigma|,M\}^{p-2}\,d\varsigma ,
\end{equation*}
so that
$\langle u_t,p\,\varphi_M\rangle
=\frac{d}{dt}\int_{\Omega_\mu}pG_M(u)\,dx$.
Substituting this together with \eqref{eq:diffusion_lower} and
\eqref{eq:nonlinear_upper} into \eqref{eq:weak_form_p} tested with
$p\,\varphi_M$, and using $|\Omega_\mu|\le V_0$ together with
$\big(2C_f/\beta\big)^{p/2}\le L^{p}$, which holds by \eqref{eq:L_constant}, we
arrive at
\begin{equation*}
    \frac{d}{dt}\int_{\Omega_\mu}pG_M(u)\,dx
    +2\alpha\int_{\Omega_\mu}|\nabla w_M|^{2}\,dx
    +\frac{p\beta}{2}\int_{\Omega_\mu}w_M^{2}\,dx
    \;\le\;\beta\,V_0\,L^{p},
\end{equation*}
an inequality in which every constant is independent of $M$.

It remains to let $M\to\infty$. Pointwise, $pG_M(u)\uparrow|u|^{p}$ and
$w_M\uparrow|u|^{p/2}$ as $M$ increases, and
$\nabla w_M\to\nabla\big(|u|^{p/2}\big)$ almost everywhere. Since the bound is
uniform in $M$, monotone convergence applied to the first and third terms and
Fatou's lemma applied to the second give both the finiteness of
$\|u(t)\|_{L^{p}(\Omega_\mu)}$ for $t\ge s$ and the asserted inequality
\eqref{eq:energy_estimate}, in the sense of distributions on $(s,\infty)$.
\eproof

\begin{rem}[On the dependence on $p$]\label{rem:p_dependence}
The forcing term on the right of \eqref{eq:energy_estimate} grows
geometrically in $p$, and this is unavoidable. The formal computation suggests
otherwise: testing with $u|u|^{p-2}$ and absorbing the lower-order term through
a pointwise inequality $C_f|u|^{p-2}\le\frac{\beta}{2}|u|^{p}+K$, one is left
with a right-hand side $pK|\Omega_\mu|$, the factor $p$ coming from multiplying
through; if $K$ is read as a constant, the growth appears to be linear. It is
not. The smallest $K$ for which that inequality holds for every real value is
\[
    K_p=\frac{2C_f}{p}\Big(\frac{2C_f(p-2)}{\beta p}\Big)^{\frac{p-2}{2}},
\]
which grows like $(2C_f/\beta)^{p/2}$ as soon as $2C_f>\beta$. The constant
obtained above, $\beta(2C_f/\beta)^{p/2}$ per unit volume, is of the
same order, so nothing is lost by proceeding through H\"older and Young rather
than pointwise.

A forcing term of this size is harmless provided it is measured on the right
scale. Were it kept as a separate additive term, it would contribute at every
step of the Moser--Alikakos iteration a factor bounded away from $1$, and the
product of those factors would diverge. The normalisation $U_k\ge L$ built into
\eqref{eq:tracking_variable} avoids this: since the forcing is $\beta V_0L^{p}$
with $L$ as in \eqref{eq:L_constant}, one has $L^{p_k}\le U_k^{\,p_k}$, the
forcing is dominated by the quantity being iterated, and the recurrence closes
multiplicatively with a factor independent of $k$.
\end{rem}

\begin{rem}\label{rem:truncation_done}
The truncation carried out in the proof of Lemma \ref{lem:Lp_energy} is the rigorous form of the formal
computation used in Proposition \ref{prop:apriori_bounds}; taking $p=q$ there
requires no change. In particular no a priori $L^{\infty}$ information is used
at any point, so the argument is not circular.
\end{rem}

Building upon Lemma \ref{lem:Lp_energy}, we set up the bootstrap iteration.
Define the parabolic gain factor $\chi=1+\frac{2}{n}$ and the geometric sequence
of exponents $p_k=q\chi^{k}$, $k\ge0$. Let $t_k=t_0+1-2^{-k}$ for $k\ge0$, with
$t_0$ the entering time of \eqref{eq:uniform_Lq}, so that
$t_{k}-t_{k-1}=r_k:=2^{-k}$ for $k\ge1$ and $t_k\uparrow t_0+1$. Finally, let
\begin{equation}\label{eq:tracking_variable}
    U_k=\max\Big\{L,\ \sup_{t\ge t_{k}}\|u(t)\|_{L^{p_k}(\Omega_\mu)}\Big\},
    \qquad k\ge0,
\end{equation}
with $L$ as in \eqref{eq:L_constant}. Since $p_0=q$ and $t_0$ is the entering
time, \eqref{eq:uniform_Lq} gives $U_0=\max\{L,C_0\}<\infty$, with $C_0$
independent of $\mu$.

\begin{prop}[Moser--Alikakos Recurrence Relation] \label{prop:moser_recurrence}
Under the standing assumptions of this section, there exists a constant
$D\ge1$, independent of $k$ and of $\mu$, such that the sequence
\eqref{eq:tracking_variable} is finite for every $k$ and satisfies
\begin{equation}\label{eq:moser_recurrence}
    U_{k+1}\;\le\;\big(D\,2^{k}\big)^{\frac{1}{p_{k+1}}}\,U_k ,
    \qquad k\ge0 .
\end{equation}
\end{prop}

\proof
Let $v_k=|u|^{p_k/2}$, so that $\|v_k\|_{L^2}^{2}=\|u\|^{p_k}_{L^{p_k}}$ and
$\|v_k\|^{2\chi}_{L^{2\chi}}=\|u\|^{p_{k+1}}_{L^{p_{k+1}}}$.

\textbf{Step 1: space--time dissipation bound.} Let $k\ge0$, let $t\ge t_k$ and
let $0<r\le1$ be such that $t-r\ge t_k$. Integrating
\eqref{eq:energy_estimate} with $p=p_k$ over $[t-r,t]$ and discarding the
nonnegative zeroth-order term gives, since $L\le U_k$ and $r\le1$,
\begin{equation}\label{eq:moser_step1}
    \int_{t-r}^{t}\|\nabla v_k(\tau)\|^{2}_{L^{2}}\,d\tau
    \;\le\;\frac{1}{2\alpha}\Big(\,\|u(t-r)\|^{p_k}_{L^{p_k}}
    +\beta V_0 L^{p_k}\Big)
    \;\le\; C_1\,U_k^{\,p_k},
    \qquad C_1:=\frac{1+\beta V_0}{2\alpha}.
\end{equation}
Note that no factor $p_k$ appears, precisely because the right-hand side of
\eqref{eq:energy_estimate} is $\beta V_0L^{p}$ rather than a multiple of $p$.

\textbf{Step 2: Gagliardo--Nirenberg interpolation.} By Theorem
\ref{thm:extension} and the corresponding inequality on $\mathbb{R}^{n}$, the
Gagliardo--Nirenberg inequality holds on the family $\{\Omega_\mu\}$ with a
constant $C_{GN}>0$ independent of $\mu$: with $\theta=\frac{n}{n+2}$,
\begin{equation*}
    \|v\|_{L^{2\chi}(\Omega_\mu)}
    \le C_{GN}\|\nabla v\|^{\theta}_{L^{2}(\Omega_\mu)}
    \|v\|^{1-\theta}_{L^{2}(\Omega_\mu)}
    +C_{GN}\|v\|_{L^{2}(\Omega_\mu)},
    \qquad v\in H^{1}(\Omega_\mu).
\end{equation*}
Raising to the power $2\chi$, using the convexity inequality
$(a+b)^{2\chi}\le2^{2\chi-1}(a^{2\chi}+b^{2\chi})$ and the identities
$2\chi\theta=2$, $2\chi(1-\theta)=\frac{4}{n}$, and noting that
$2^{2\chi-1}=2^{1+4/n}$ depends only on $n$, we get
$\|v_k\|^{2\chi}_{L^{2\chi}}\le
2^{1+4/n}C_{GN}^{2\chi}\big(\|\nabla v_k\|^{2}_{L^{2}}\|v_k\|^{4/n}_{L^{2}}
+\|v_k\|^{2\chi}_{L^{2}}\big)$.
 Since
$\|v_k\|^{4/n}_{L^{2}}=\|u\|^{2p_k/n}_{L^{p_k}}\le U_k^{2p_k/n}$ and
$p_k+\frac{2p_k}{n}=p_k\chi=p_{k+1}$, integrating over $[t-r,t]$ and inserting
\eqref{eq:moser_step1} yields
\begin{equation}\label{eq:moser_step2}
    \int_{t-r}^{t}\|u(\tau)\|^{p_{k+1}}_{L^{p_{k+1}}}\,d\tau
    \;\le\; C_4\,U_k^{\,p_{k+1}},
    \qquad C_4:=2^{1+4/n}C_{GN}^{2\chi}(C_1+1),
\end{equation}
with $C_4$ independent of $k$ and of $\mu$.

\textbf{Step 3: uniform Gronwall.} Let
$y_{k+1}(t)=\int_{\Omega_\mu}|u(t)|^{p_{k+1}}\,dx$. By
\eqref{eq:energy_estimate} with $p=p_{k+1}$ --- legitimate because
\eqref{eq:moser_step2} gives $\|u(s)\|_{L^{p_{k+1}}(\Omega_\mu)}<\infty$ for
almost every $s$ in the window --- we have
$y_{k+1}'\le\beta V_0L^{p_{k+1}}$, so the Uniform Gronwall Lemma (Lemma
\ref{lem:gronwall}) applied on the window $[t-r_{k+1},t]$, $t\ge t_{k+1}$,
gives
\begin{equation*}
    y_{k+1}(t)\;\le\;\frac{1}{r_{k+1}}\int_{t-r_{k+1}}^{t}y_{k+1}(\tau)\,d\tau
    +\beta V_0L^{p_{k+1}}r_{k+1} .
\end{equation*}
Here $t-r_{k+1}\ge t_{k+1}-r_{k+1}=t_k$, so the whole window lies in the range
$[t_k,\infty)$ on which $U_k$ controls $\|u\|_{L^{p_k}(\Omega_\mu)}$. Using
\eqref{eq:moser_step2} with $r=r_{k+1}$, $1/r_{k+1}=2^{k+1}$, $r_{k+1}\le1$ and
$L\le U_k$, we obtain
\begin{equation*}
    y_{k+1}(t)\;\le\;2^{k+1}C_4\,U_k^{\,p_{k+1}}+\beta V_0\,U_k^{\,p_{k+1}}
    \;\le\; D\,2^{k}\,U_k^{\,p_{k+1}},
    \qquad D:=\max\{1,\,2(C_4+\beta V_0)\} .
\end{equation*}

Taking the supremum over $t\ge t_{k+1}$ and extracting the $p_{k+1}$-th root
gives
\begin{equation*}
    \sup_{t\ge t_{k+1}}\|u(t)\|_{L^{p_{k+1}}(\Omega_\mu)}
    \ \le\ (D2^{k})^{1/p_{k+1}}\,U_k .
\end{equation*}
Since $L\le U_k$ by \eqref{eq:tracking_variable} and $D2^{k}\ge1$, both terms
in the maximum defining $U_{k+1}$ are bounded by the right-hand side, and
\eqref{eq:moser_recurrence} follows. Since $U_0<\infty$ by
\eqref{eq:uniform_Lq}, induction gives $U_k<\infty$ for every $k$.
\eproof

\begin{theorem}[Uniform $L^\infty$ Bound]\label{thm:uniform_linfty}
Suppose that $\{\Omega_{\mu}\}_{\mu \in \Lambda}$ is a uniform Jones family of
domains with uniformly bounded volume, and assume that the strict ellipticity
condition (Hypothesis \ref{hyp:elliptic_condition}), the subcritical growth
condition (Hypothesis \ref{hyp:nonlinear_growth}) and the dissipativity
condition (Hypothesis \ref{hyp:nonlinear_dissipativity}) hold. Then there exists
a constant $C_\infty>0$, independent of $\mu$, such that every global weak
solution $u(t)$ originating in the $L^{q}(\Omega_\mu)$ absorbing ball satisfies
\begin{equation*}
    \sup_{t\ge t_0+1}\|u(t)\|_{L^{\infty}(\Omega_\mu)}\le C_\infty .
\end{equation*}
Consequently the family of global attractors $\{\mathscr{A}_\mu\}$ is uniformly
bounded in $L^{\infty}(\Omega_\mu)$ by $C_\infty$.
\end{theorem}

\proof
Iterating the recurrence of Proposition \ref{prop:moser_recurrence} from
$k=0$ to $k=m$,
\begin{equation*}
    U_{m+1}\;\le\;U_0\prod_{k=0}^{m}\big(D2^{k}\big)^{\frac{1}{p_{k+1}}}
    \;=\;U_0\,\exp\Big(\sum_{k=0}^{m}
    \frac{\log D+k\log 2}{q\,\chi^{\,k+1}}\Big).
\end{equation*}
Since $\chi=1+\frac{2}{n}>1$, the series
$\sum_{k\ge0}(\log D+k\log2)\,\chi^{-(k+1)}$ converges, so that
\begin{equation*}
    \sup_{k\ge0}U_k\;\le\;\Gamma\,U_0=:C_\infty,
    \qquad
    \Gamma:=\exp\Big(\frac{1}{q}\sum_{k\ge0}
    \frac{\log D+k\log 2}{\chi^{\,k+1}}\Big)<\infty ,
\end{equation*}
and $\Gamma$, $U_0=\max\{L,C_0\}$ --- hence $C_\infty$ --- depend only on
$n$, $q$, $\alpha$, $\beta$, $C$, $V_0$ and $C_{GN}$, all of which are
independent of $\mu$.

Fix now $t\ge t_0+1$. Since $t\ge t_k$ for every $k$, the definition
\eqref{eq:tracking_variable} gives $\|u(t)\|_{L^{p_k}(\Omega_\mu)}\le U_k\le
C_\infty$ for all $k$, and since $p_k\to\infty$,
\begin{equation*}
    \|u(t)\|_{L^{\infty}(\Omega_\mu)}
    \le\liminf_{k\to\infty}\|u(t)\|_{L^{p_k}(\Omega_\mu)}\le C_\infty .
\end{equation*}

Finally, let $w\in\mathscr{A}_\mu$. By the strict invariance of the attractor
there is, for each $T>0$, an element $v\in\mathscr{A}_\mu$ with
$w=S_\mu(T)v$; choosing $T\ge t_0+1$ and noting that $\mathscr{A}_\mu$ is
contained in the $L^{q}(\Omega_\mu)$ absorbing ball, the bound just proved gives
$\|w\|_{L^{\infty}(\Omega_\mu)}\le C_\infty$.
\eproof

\section{Upper Semicontinuity of Attractors}
\label{sec:upper_semicontinuity}

In this section we analyse the robustness of the global attractors as the domain
perturbation parameter vanishes. Since the attractors $\mathscr{A}_\mu$ and
$\mathscr{A}_0$ live in phase spaces attached to different domains, their
distance cannot be measured by the Hausdorff semidistance of a fixed space, and
a framework for comparing functions across varying domains is needed.

Our starting point is the uniform $L^{\infty}$ bound of Theorem
\ref{thm:uniform_linfty}. It is what makes the whole section elementary: on the
attractors the nonlinearity is globally Lipschitz, so no fractional machinery,
and no singular Gronwall kernel, is required. The natural topology for the
argument is that of $L^{2}$; the corresponding statement in $H^{s}(\Omega_\mu)$
for every $s<1$ is then recovered by interpolation, in Corollary
\ref{cor:semicontinuity_Hs}, at no extra cost. The endpoint $s=1$ is genuinely
out of reach of the present framework, for reasons discussed in Remark
\ref{rem:open_s1}.

\medskip
\noindent\textbf{Standing Assumptions for Section \ref{sec:upper_semicontinuity}.}
Throughout this section:
\begin{itemize}
    \item[(H1)] $\{\Omega_\mu\}_{\mu\ge0}$ is a uniform Jones family with
    uniformly bounded volume;
    \item[(H2)] $|\Omega_\mu\triangle\Omega_0|\to0$ as $\mu\to0$;
    \item[(H3)] Hypotheses \ref{hyp:elliptic_condition},
    \ref{hyp:nonlinear_growth} and \ref{hyp:nonlinear_dissipativity} hold.
\end{itemize}
We write $\mathcal{A}_\mu$ for the shifted elliptic operator of
\eqref{def:lin_operator}, $S_\mu(t)$ for the semigroup it generates together
with the nonlinearity, and $\mathscr{A}_\mu$ for the corresponding global
attractor, and we denote by $C_\infty$ and $\rho_{H^1}$ the uniform bounds of
Theorem \ref{thm:uniform_linfty} and Proposition \ref{prop:absorbing_X1}.

No containment hypothesis on the domains is needed: by Lemma
\ref{lem:uniform_diameter} the diameters of the $\Omega_\mu$ are bounded by a
constant $d_0=d_0(\delta_0,\epsilon,n,V_0)$, while (H2) and $|\Omega_0|>0$ give
$|\Omega_\mu\cap\Omega_0|>0$, hence $\Omega_\mu\cap\Omega_0\neq\emptyset$, for all
sufficiently small $\mu$. Consequently
\begin{equation}\label{eq:holdall}
    \Omega_\mu\ \subset\ D:=\big\{x\in\mathbb{R}^{n}:\ \operatorname{dist}(x,\Omega_0)<d_0+1\big\}
\end{equation}
for every sufficiently small $\mu$, and $D$ is a fixed bounded open set.

\subsection{Connecting maps and convergence in varying spaces}

By Theorem \ref{thm:extension} the extension operators $\mathfrak{E}_\mu$ are
bounded, uniformly in $\mu$, on $L^{2}$, on $H^{1}$ and on $L^{\infty}$, with a
common bound $K$. We define the connecting maps
\begin{equation}\label{eq:connecting_maps}
    E_\mu u=(\mathfrak{E}_0u)\big|_{\Omega_\mu},
    \qquad
    F_\mu v=(\mathfrak{E}_\mu v)\big|_{\Omega_0},
\end{equation}
which are therefore bounded from $L^{2}(\Omega_0)$ to $L^{2}(\Omega_\mu)$, and
from $H^{1}(\Omega_0)$ to $H^{1}(\Omega_\mu)$, by $K$, and likewise for $F_\mu$
with the roles of the domains exchanged.

\begin{definition}[Convergence in varying spaces]\label{def:varying_convergence}
Let $u_\mu\in L^{2}(\Omega_\mu)$ and $u_0\in L^{2}(\Omega_0)$, and denote by
$\widetilde{v}$ the extension of a function $v$ by zero outside its domain.
\begin{enumerate}
    \item[(i)] $u_\mu\to u_0$ \emph{strongly} if
    $\|u_\mu-E_\mu u_0\|_{L^{2}(\Omega_\mu)}\to0$ as $\mu\to0$.
    \item[(ii)] $u_\mu\rightharpoonup u_0$ \emph{weakly} if
    $\sup_{\mu}\|u_\mu\|_{L^{2}(\Omega_\mu)}<\infty$ and
    $\widetilde{u_\mu}\rightharpoonup\widetilde{u_0}$ weakly in
    $L^{2}(\mathbb{R}^{n})$.
\end{enumerate}
\end{definition}

\begin{rem}\label{rem:pivot}
Both notions are formulated in the pivot space $L^{2}$, as required by the Mosco
framework for symmetric forms. Strong convergence in the sense of (i) does not
depend on the particular extension operator used: if $\mathfrak{E}'_0$ is
another uniformly bounded extension, then
$w:=\mathfrak{E}_0u_0-\mathfrak{E}'_0u_0$ vanishes a.e.\ on $\Omega_0$, being
the difference of two extensions of the same function, and is therefore
supported in $\mathbb{R}^{n}\setminus\Omega_0$. Consequently
\begin{equation*}
    \|w\|_{L^{2}(\Omega_\mu)}=\|w\|_{L^{2}(\Omega_\mu\setminus\Omega_0)} ,
\end{equation*}
and since $w\in L^{2}(\mathbb{R}^{n})$ is fixed while
$|\Omega_\mu\setminus\Omega_0|\to0$ by (H2), the absolute continuity of the
integral gives $\|w\|_{L^{2}(\Omega_\mu)}\to0$ as $\mu\to0$, for each fixed
$u_0$.
\end{rem}

\subsection{Mosco convergence of the forms}

For $\mu\ge0$ let $a_\mu$ be the symmetric bilinear form of
\eqref{def:lin_operator},
\begin{equation}\label{eq:bilinear_form}
    a_\mu(u,v)=\int_{\Omega_\mu}\Big(A(x)\nabla u\cdot\nabla v
    +\tfrac{\beta}{2}\,uv\Big)dx,
    \qquad u,v\in H^{1}(\Omega_\mu),
\end{equation}
extended by $a_\mu(u,u)=+\infty$ for $u\in L^{2}(\Omega_\mu)\setminus
H^{1}(\Omega_\mu)$. By Hypothesis \ref{hyp:elliptic_condition} the forms are
coercive uniformly in $\mu$.

\begin{theorem}[Mosco convergence]\label{thm:mosco}
Assume (H1) and (H2), and Hypothesis \ref{hyp:elliptic_condition}. Then
$a_\mu\to a_0$ in the sense of Mosco, that is:
\begin{enumerate}
    \item[(M1)] if $u_\mu\rightharpoonup u_0$ weakly in the sense of Definition
    \ref{def:varying_convergence}, then
    $\liminf_{\mu\to0}a_\mu(u_\mu,u_\mu)\ge a_0(u_0,u_0)$;
    \item[(M2)] for every $v_0\in H^{1}(\Omega_0)$ the sequence
    $v_\mu:=E_\mu v_0$ converges strongly to $v_0$ in the sense of Definition
    \ref{def:varying_convergence}(i) and satisfies
    \begin{equation}\label{eq:recovery_limit}
        \lim_{\mu\to0}a_\mu\big(E_\mu v_0,E_\mu v_0\big)=a_0(v_0,v_0) ;
    \end{equation}
    in particular $\limsup_{\mu\to0}a_\mu(v_\mu,v_\mu)\le a_0(v_0,v_0)$,
    which is the recovery condition required by Mosco convergence. Note that
    \eqref{eq:recovery_limit} is stronger than that condition in two respects:
    the recovery sequence is canonical, and the limit exists. Both are used
    below.
\end{enumerate}
\end{theorem}

\proof
\textbf{(M1).} We may assume $\liminf_{\mu\to0}a_\mu(u_\mu,u_\mu)<\infty$, for
otherwise there is nothing to prove, and, passing to a subsequence realising the
$\liminf$, that $\sup_\mu a_\mu(u_\mu,u_\mu)<\infty$. By uniform coercivity,
$\sup_\mu\|u_\mu\|_{H^{1}(\Omega_\mu)}<\infty$, so by Theorem
\ref{thm:extension} the functions $U_\mu:=\mathfrak{E}_\mu u_\mu$ are bounded in
$H^{1}(\mathbb{R}^{n})$; extracting a further subsequence, $U_\mu\rightharpoonup
U$ weakly in $H^{1}(\mathbb{R}^{n})$ and, by the Rellich--Kondrachov theorem,
$U_\mu\to U$ strongly in $L^{2}_{\mathrm{loc}}(\mathbb{R}^{n})$.

We first identify $U$ on $\Omega_0$. Decompose
\begin{equation*}
    \widetilde{u_\mu}-\chi_{\Omega_0}U
    =\chi_{\Omega_\mu}\big(U_\mu-U\big)
    +\big(\chi_{\Omega_\mu}-\chi_{\Omega_0}\big)U .
\end{equation*}
Since $\Omega_\mu\subset D$ and $\|\chi_{\Omega_\mu}\|_{L^{\infty}}\le1$, the
first term is bounded in $L^{2}(\mathbb{R}^{n})$ by $\|U_\mu-U\|_{L^{2}(D)}$,
which tends to $0$ by the strong convergence just noted; the second tends to $0$
by (H2) and Vitali's convergence theorem, the function $U\in
L^{2}(\mathbb{R}^{n})$ being fixed. Hence
$\widetilde{u_\mu}=\chi_{\Omega_\mu}U_\mu\to\chi_{\Omega_0}U$ strongly, and in
particular weakly, in $L^{2}(\mathbb{R}^{n})$. Since
$\widetilde{u_\mu}\rightharpoonup\widetilde{u_0}$ by Definition
\ref{def:varying_convergence}(ii) and weak limits in $L^{2}(\mathbb{R}^{n})$ are
unique, we conclude $\widetilde{u_0}=\chi_{\Omega_0}U$, that is,
\begin{equation}\label{eq:identify_limit}
    U|_{\Omega_0}=u_0 .
\end{equation}

Let now $U_0:=\mathfrak{E}_0u_0\in H^{1}(\mathbb{R}^{n})$. Since $A(x)$ is
positive semidefinite almost everywhere and $\beta>0$,
\begin{equation*}
    \int_{\mathbb{R}^{n}}\chi_{\Omega_\mu}
    \Big(A(x)\nabla(U_\mu-U_0)\cdot\nabla(U_\mu-U_0)
    +\tfrac{\beta}{2}|U_\mu-U_0|^{2}\Big)dx\;\ge\;0 ,
\end{equation*}
and expanding the square and isolating $a_\mu(u_\mu,u_\mu)$ we obtain
\begin{align}\label{eq:mosco_expansion}
    a_\mu(u_\mu,u_\mu)
    &\ge 2\int_{\mathbb{R}^{n}}\chi_{\Omega_\mu}
      \Big(A(x)\nabla U_\mu\cdot\nabla U_0+\tfrac{\beta}{2}U_\mu U_0\Big)dx
      \nonumber\\
    &\quad-\int_{\mathbb{R}^{n}}\chi_{\Omega_\mu}
      \Big(A(x)\nabla U_0\cdot\nabla U_0+\tfrac{\beta}{2}|U_0|^{2}\Big)dx .
\end{align}
By (H2) we have $\chi_{\Omega_\mu}\to\chi_{\Omega_0}$ in measure and, being
uniformly bounded, $\chi_{\Omega_\mu}A(x)\nabla U_0\to\chi_{\Omega_0}A(x)\nabla
U_0$ and $\chi_{\Omega_\mu}U_0\to\chi_{\Omega_0}U_0$ strongly in
$L^{2}(\mathbb{R}^{n})$, by Vitali's convergence theorem. Since
$U_\mu\rightharpoonup U$ weakly in $H^{1}(\mathbb{R}^{n})$, the first integral in
\eqref{eq:mosco_expansion} is a pairing of a weakly convergent with a strongly
convergent sequence, and by \eqref{eq:identify_limit} its limit is
$2a_0(u_0,u_0)$; the second integral converges to $a_0(u_0,u_0)$. Since the
right-hand side of \eqref{eq:mosco_expansion} therefore converges, taking the
$\liminf$ on the left gives
\begin{equation*}
    \liminf_{\mu\to0}a_\mu(u_\mu,u_\mu)
    \;\ge\;2a_0(u_0,u_0)-a_0(u_0,u_0)=a_0(u_0,u_0),
\end{equation*}
as claimed.

\textbf{(M2).} Take $v_\mu:=E_\mu v_0$, so that
$\|v_\mu-E_\mu v_0\|_{L^{2}(\Omega_\mu)}=0$ and the strong convergence holds
trivially. Writing $V_0:=\mathfrak{E}_0v_0\in H^{1}(\mathbb{R}^{n})$,
\begin{equation*}
    a_\mu(v_\mu,v_\mu)-a_0(v_0,v_0)
    =\int_{\Omega_\mu\setminus\Omega_0}g\,dx-\int_{\Omega_0\setminus\Omega_\mu}g\,dx,
    \qquad
    g:=A(x)\nabla V_0\cdot\nabla V_0+\tfrac{\beta}{2}|V_0|^{2}\in L^{1}(\mathbb{R}^{n}).
\end{equation*}
By the absolute continuity of the Lebesgue integral and (H2), both terms tend to
$0$, so in fact $\lim_{\mu\to0}a_\mu(v_\mu,v_\mu)=a_0(v_0,v_0)$.
\eproof

\begin{cor}[Convergence of resolvents and of the linear semigroups]
\label{cor:resolvent_convergence}
Assume (H1), (H2) and Hypothesis \ref{hyp:elliptic_condition}. If
$f_\mu\in L^{2}(\Omega_\mu)$ and $f_0\in L^{2}(\Omega_0)$ satisfy
$f_\mu\to f_0$ strongly in the sense of Definition
\ref{def:varying_convergence}(i), then the solutions
$u_\mu=\mathcal{A}_\mu^{-1}f_\mu$ satisfy
\begin{equation}\label{eq:resolvent_conv}
    \big\|u_\mu-E_\mu u_0\big\|_{H^{1}(\Omega_\mu)}\longrightarrow 0,
    \qquad u_0=\mathcal{A}_0^{-1}f_0 .
\end{equation}
Consequently, for every $t\ge0$ and every $\phi_0\in L^{2}(\Omega_0)$ and
$\phi_\mu\to\phi_0$ strongly,
\begin{equation}\label{eq:semigroup_conv}
    \big\|e^{-\mathcal{A}_\mu t}\phi_\mu
    -E_\mu e^{-\mathcal{A}_0 t}\phi_0\big\|_{L^{2}(\Omega_\mu)}
    \longrightarrow 0,
\end{equation}
\end{cor}

\proof
The first assertion is the standard consequence of Mosco convergence for
symmetric, uniformly coercive forms; we recall the argument, since it also
yields convergence in the stronger $H^{1}$ topology.

Testing $a_\mu(u_\mu,\phi)=(f_\mu,\phi)_{L^{2}}$ with $\phi=u_\mu$ and using
uniform coercivity gives $\sup_\mu\|u_\mu\|_{H^{1}(\Omega_\mu)}<\infty$;
extracting extensions as in the proof of Theorem \ref{thm:mosco}, a subsequence
converges weakly, in the sense of Definition \ref{def:varying_convergence}(ii),
to some $u^{*}\in H^{1}(\Omega_0)$.

Recall that, $a_\mu$ being symmetric, continuous and coercive, the weak solution
$u_\mu$ is precisely the minimiser over $H^{1}(\Omega_\mu)$ of
$J_\mu(w)=\frac12a_\mu(w,w)-(f_\mu,w)_{L^{2}}$: writing $w=u_\mu+v$ and using
$a_\mu(u_\mu,v)=(f_\mu,v)_{L^{2}}$,
\begin{equation}\label{eq:J_identity}
    J_\mu(w)=J_\mu(u_\mu)+\tfrac12a_\mu(v,v)
    \ \ge\ J_\mu(u_\mu)+\tfrac{\alpha_0}{2}\,
    \|w-u_\mu\|^{2}_{H^{1}(\Omega_\mu)} ,
\end{equation}
so that the minimiser exists, is unique and coincides with the weak solution.
By \eqref{eq:recovery_limit} applied to $u_0$, the sequence $v_\mu:=E_\mu u_0$
converges strongly to $u_0$, whence
$\limsup_\mu J_\mu(u_\mu)\le\limsup_\mu J_\mu(v_\mu)\le J_0(u_0)$; by (M1) and
the strong convergence of $f_\mu$, $J_0(u^{*})\le\liminf_\mu J_\mu(u_\mu)$.
Therefore $J_0(u^{*})\le J_0(u_0)$, and since $J_0$ has a unique minimiser,
$u^{*}=u_0$; as the limit does not depend on the subsequence, the whole family
converges. The same chain of inequalities forces
$\lim_\mu J_\mu(u_\mu)=J_0(u_0)$ and hence
$\lim_\mu a_\mu(u_\mu,u_\mu)=a_0(u_0,u_0)$.

For \eqref{eq:resolvent_conv} it now suffices to take $w=E_\mu u_0$ in
\eqref{eq:J_identity}:
\begin{equation*}
    \frac{\alpha_0}{2}\,\big\|u_\mu-E_\mu u_0\big\|^{2}_{H^{1}(\Omega_\mu)}
    \ \le\ J_\mu(E_\mu u_0)-J_\mu(u_\mu) .
\end{equation*}
By \eqref{eq:recovery_limit} with $v_0=u_0$ we have
$a_\mu(E_\mu u_0,E_\mu u_0)\to a_0(u_0,u_0)$, and
$(f_\mu,E_\mu u_0)_{L^{2}}\to(f_0,u_0)_{L^{2}}$ since $f_\mu\to f_0$ and
$E_\mu u_0\to u_0$ strongly; hence $J_\mu(E_\mu u_0)\to J_0(u_0)$. Together
with $J_\mu(u_\mu)\to J_0(u_0)$, established above, the right-hand side tends
to $0$.

Finally, \eqref{eq:semigroup_conv} follows from \eqref{eq:resolvent_conv}
through the equivalence, valid in the framework of converging Hilbert spaces,
between Mosco convergence of the forms, strong convergence of the resolvents and
strong convergence of the semigroups; see Kuwae--Shioya
\cite[Thm.~2.4]{KuwaeShioya2003}, where conditions (1), (2) and (3) are
precisely these three. Their framework is the one used here: the Mosco
convergence of their Definition 2.11 is conditions (M1)--(M2) of Theorem
\ref{thm:mosco}, and their strong convergence of vectors in varying spaces
(Definition 2.4 of \cite{KuwaeShioya2003}) is Definition
\ref{def:varying_convergence}(i) above, the normalisation
$\|E_\mu u_0\|_{L^{2}(\Omega_\mu)}\to\|u_0\|_{L^{2}(\Omega_0)}$ required of the
connecting maps being immediate from (H2) and the absolute continuity of the
integral of $|\mathfrak{E}_0u_0|^{2}$. For a fixed Hilbert space the
corresponding statement is classical and goes back to Mosco and to Attouch
\cite{Attouch1984}. We note that \cite[Thm.~2.4]{KuwaeShioya2003} yields the
convergence in \eqref{eq:semigroup_conv} for each fixed $t$; no uniformity in
$t$ is claimed, and none is needed below.
\eproof

\subsection{Convergence of the nonlinear flows}

The nonlinearity does not commute with the connecting maps: in general
$\mathfrak{E}_0(f(w))\neq f(\mathfrak{E}_0w)$, since the extension is not a
pointwise operation. The discrepancy is, however, supported in
$\Omega_\mu\setminus\Omega_0$, and on $L^{\infty}$-bounded sets it is
controlled by (H2).

\begin{lemma}[Almost-commutation of the connecting maps]\label{lem:commutation}
Assume (H1) and (H2), and let $R>0$. Then
\begin{equation*}
    \sup\Big\{\big\|E_\mu f(w)-f(E_\mu w)\big\|_{L^{2}(\Omega_\mu)}
    \;:\;w\in L^{\infty}(\Omega_0),\ \|w\|_{L^{\infty}}\le R\Big\}
    \longrightarrow 0
    \qquad\text{as }\mu\to0 .
\end{equation*}
\end{lemma}

\proof
On $\Omega_\mu\cap\Omega_0$ we have $\mathfrak{E}_0w=w$ and
$\mathfrak{E}_0(f(w))=f(w)$, so the two functions coincide there and the
difference is supported in $\Omega_\mu\setminus\Omega_0$. By Theorem
\ref{thm:extension} the operator $\mathfrak{E}_0$ is bounded on $L^{\infty}$ by
$K$, so that $\|\mathfrak{E}_0w\|_{L^{\infty}}\le KR$ and, $f$ being continuous,
\begin{equation*}
    \big\|E_\mu f(w)\big\|_{L^{\infty}}\le K\!\!\sup_{|\sigma|\le R}\!|f(\sigma)| ,
    \qquad
    \big\|f(E_\mu w)\big\|_{L^{\infty}}\le\!\!\sup_{|\sigma|\le KR}\!|f(\sigma)|
    =:M_{R} .
\end{equation*}
Let  $C(R):=K\sup_{|\sigma|\le R}|f(\sigma)|+M_R$. Then
$\|E_\mu f(w)-f(E_\mu w)\|_{L^{2}(\Omega_\mu)}
\le C(R)\,|\Omega_\mu\setminus\Omega_0|^{1/2}\to0$ by (H2), uniformly in $w$.
\eproof

Fix once and for all a function $\psi\in C_c^{\infty}(\mathbb{R}^{n})$ with
$0\le\psi\le1$ and $\psi\equiv1$ on $\overline{D}$, where $D$ is the fixed
bounded open set of \eqref{eq:holdall}. We emphasise that $\psi$ is smooth:
replacing it by the characteristic function of $D$ would introduce a jump across
$\partial D$ and destroy the $H^{1}(\mathbb{R}^{n})$ regularity on which the
next lemma rests.

\begin{lemma}[Collective precompactness]\label{lem:collective_compactness}
Assume (H1)--(H3), let $\psi$ be as above and let
\begin{equation}\label{eq:K_set}
    \mathbb{K}:=\overline{\Big\{\psi\,\mathfrak{E}_\mu v\;:\;
    v\in\mathscr{A}_\mu,\ \mu\in(0,\mu_0]\Big\}}^{\,L^{2}(\mathbb{R}^{n})} .
\end{equation}
Then $\mathbb{K}$ is a compact subset of $L^{2}(\mathbb{R}^{n})$, bounded in
$H^{1}(\mathbb{R}^{n})$ by a constant depending only on
$\|\psi\|_{W^{1,\infty}}$, $K$ and $\rho_{H^1}$, and in $L^{\infty}$ by
$KC_\infty$; moreover $(\psi\,\mathfrak{E}_\mu v)\big|_{\Omega_\mu}=v$ for
every $v\in\mathscr{A}_\mu$.
\end{lemma}

\proof
The last identity holds because $\psi\equiv1$ on $\overline{D}\supset\Omega_\mu$
and $\mathfrak{E}_\mu v=v$ on $\Omega_\mu$.

By Proposition \ref{prop:absorbing_X1} and Proposition
\ref{prop:truncation_linfty} the attractors are bounded in $H^{1}(\Omega_\mu)$
by $\rho_{H^1}$ and in $L^{\infty}(\Omega_\mu)$ by $C_\infty$, uniformly in
$\mu$; by Theorem \ref{thm:extension} the family $\{\psi\,\mathfrak{E}_\mu v\}$
is therefore bounded in $H^{1}(\mathbb{R}^{n})$ and in
$L^{\infty}(\mathbb{R}^{n})$, uniformly in $\mu$, with supports contained in the
fixed compact set $\operatorname{supp}\psi$. Here $C_\infty$ may be taken to be
the truncation constant $L$ of \eqref{eq:L_constant}, so that nothing in this
section depends on the Moser--Alikakos iteration of Section \ref{sec:uniform_linfty}.
The Rellich--Kondrachov theorem applied on a fixed ball containing
$\operatorname{supp}\psi$ gives precompactness in $L^{2}(\mathbb{R}^{n})$, so
that $\mathbb{K}$, being the $L^{2}$-closure of a precompact set, is compact.

Finally, both bounds pass to the $L^{2}$-closure. Let $w\in\mathbb{K}$ and let
$(w_k)$ be elements of the set in \eqref{eq:K_set} with $w_k\to w$ in
$L^{2}(\mathbb{R}^{n})$. The sequence $(w_k)$ is bounded in
$H^{1}(\mathbb{R}^{n})$, which is reflexive, so a subsequence converges weakly
in $H^{1}(\mathbb{R}^{n})$; its weak limit is also a weak $L^{2}$ limit and
therefore equals $w$. Hence $w\in H^{1}(\mathbb{R}^{n})$ and, by the weak lower
semicontinuity of the norm, $\|w\|_{H^{1}(\mathbb{R}^{n})}\le
\liminf_k\|w_k\|_{H^{1}(\mathbb{R}^{n})}$. Extracting a further subsequence
converging almost everywhere, we also get $|w|\le KC_\infty$ almost everywhere.
Thus $\mathbb{K}$ satisfies the same two bounds.
\eproof

\begin{lemma}[Uniform convergence of the flows]\label{lem:unif_conv}
Assume (H1)--(H3) and let $\mathbb{K}$ be as in \eqref{eq:K_set}. For every
fixed $T>0$,
\begin{equation}\label{eq:unif_flow}
    \lim_{\mu\to0}\ \sup_{u\in\mathbb{K}}
    \big\|S_\mu(T)R_\mu u-E_\mu S_0(T)R_0u\big\|_{L^{2}(\Omega_\mu)}=0,
\end{equation}
where $R_\mu u=u|_{\Omega_\mu}$ and $R_0u=u|_{\Omega_0}$.
\end{lemma}

\proof
Fix $u\in\mathbb{K}$ and write $v_\mu(t)=S_\mu(t)R_\mu u$,
$v_0(t)=S_0(t)R_0u$.

By Lemma \ref{lem:collective_compactness} we have
$\|R_\mu u\|_{L^{\infty}(\Omega_\mu)}\le KC_\infty$ and
$\|R_0u\|_{L^{\infty}(\Omega_0)}\le KC_\infty$, uniformly in $\mu$ and in
$u\in\mathbb{K}$. By Proposition \ref{cor:absorption_rate} with $\tau=0$ the
quantity $\|(|v_\mu(t)|-L)^{+}\|_{L^{\infty}(\Omega_\mu)}$ is nonincreasing in
$t$, whence
\begin{equation}\label{eq:orbit_linfty}
    \|v_\mu(t)\|_{L^{\infty}(\Omega_\mu)}\le R:=\max\{L,\,KC_\infty\},
    \qquad t\ge0,
\end{equation}
and likewise $\|v_0(t)\|_{L^{\infty}(\Omega_0)}\le R$; the constant $R$ is
independent of $\mu$, of $u\in\mathbb{K}$ and of $t$. Since
$\|E_\mu v_0(t)\|_{L^{\infty}(\Omega_\mu)}\le KR$, we may set
$\ell:=\operatorname{Lip}\big(f;[-KR,KR]\big)$.

We shall use twice the following elementary fact: if $(T_\mu)$ is a family of
linear operators with $\sup_\mu\|T_\mu\|\le c$ converging strongly to $0$, and
$\mathcal{C}$ is a compact subset of their common domain, then
$\sup_{z\in\mathcal{C}}\|T_\mu z\|\to0$. Indeed, given $\epsilon>0$, cover
$\mathcal{C}$ by finitely many balls of radius $\frac{\epsilon}{2c}$ centred at
$z_1,\dots,z_m$ and choose $\mu$ so small that
$\|T_\mu z_i\|<\frac{\epsilon}{2}$ for every $i$; for $z\in\mathcal{C}$ and
$z_i$ with $\|z-z_i\|<\frac{\epsilon}{2c}$ we get
$\|T_\mu z\|\le\|T_\mu(z-z_i)\|+\|T_\mu z_i\|<\epsilon$.

The two compact sets to which it will be applied are $R_0\mathbb{K}$, which is
compact in $L^{2}(\Omega_0)$ as the image of the compact set $\mathbb{K}$ under
the bounded linear map $R_0$, and
\begin{equation*}
    \mathcal{G}:=\big\{F\big(S_0(s)R_0u\big)\ :\ s\in[0,T],\ u\in\mathbb{K}\big\}
    \subset L^{2}(\Omega_0).
\end{equation*}
That $\mathcal{G}$ is compact follows as soon as $(s,w)\mapsto F(S_0(s)w)$ is
continuous on $[0,T]\times R_0\mathbb{K}$ for the $L^{2}$ topology. This is the
case: by \eqref{eq:orbit_linfty} the orbits issuing from $R_0\mathbb{K}$ remain
bounded by $R$ in $L^{\infty}(\Omega_0)$, and since $e^{-\mathcal{A}_0\sigma}$ is
a contraction on $L^{2}(\Omega_0)$ and $f$ is Lipschitz with constant $\ell$
there, the variation of constants formula together with Gronwall's inequality
gives
\begin{equation*}
    \|S_0(s)w_1-S_0(s)w_2\|_{L^{2}(\Omega_0)}
    \le e^{\ell s}\|w_1-w_2\|_{L^{2}(\Omega_0)},
    \qquad w_1,w_2\in R_0\mathbb{K},\ s\in[0,T],
\end{equation*}
a Lipschitz bound uniform in $s$; combined with the continuity of
$s\mapsto S_0(s)w$ in $L^{2}(\Omega_0)$ for fixed $w$ this yields joint
continuity, and $F$ is Lipschitz in the $L^{2}$ norm on $L^{\infty}$-bounded
sets. Thus $\mathcal{G}$ is the continuous image of the compact set
$[0,T]\times R_0\mathbb{K}$. We stress that only the $L^{2}$ topology is
involved, which is what makes the compactness of $\mathbb{K}$ in
$L^{2}(\mathbb{R}^{n})$ sufficient.

By the variation of constants formula, for $t\in[0,T]$,
\begin{align*}
    v_\mu(t)-E_\mu v_0(t)
    &=\Big(e^{-\mathcal{A}_\mu t}R_\mu u-E_\mu e^{-\mathcal{A}_0t}R_0u\Big)\\
    &\quad+\int_0^{t}e^{-\mathcal{A}_\mu(t-s)}
      \Big(F(v_\mu(s))-F\big(E_\mu v_0(s)\big)\Big)ds\\
    &\quad+\int_0^{t}e^{-\mathcal{A}_\mu(t-s)}
      \Big(F\big(E_\mu v_0(s)\big)-E_\mu F(v_0(s))\Big)ds\\
    &\quad+\int_0^{t}\Theta_\mu(t-s)\,F(v_0(s))\,ds
    \;=:\;\mathrm{I}(t)+\mathrm{II}(t)+\mathrm{III}(t)+\mathrm{IV}(t),
\end{align*}
where
$\Theta_\mu(\sigma):=e^{-\mathcal{A}_\mu\sigma}E_\mu-E_\mu
e^{-\mathcal{A}_0\sigma}$ satisfies $\|\Theta_\mu(\sigma)\|\le2K$ and, by
\eqref{eq:semigroup_conv}, converges strongly to $0$ for each fixed
$\sigma\ge0$. Write
$\phi(t):=\|v_\mu(t)-E_\mu v_0(t)\|_{L^{2}(\Omega_\mu)}$.

Since $\mathcal{A}_\mu$ is nonnegative and self-adjoint, $e^{-\mathcal{A}_\mu
t}$ is a contraction on $L^{2}(\Omega_\mu)$, so by the choice of $\ell$
\begin{equation}\label{eq:gronwall_kernel}
    \|\mathrm{II}(t)\|_{L^{2}(\Omega_\mu)}\le\ell\int_0^{t}\phi(s)\,ds .
\end{equation}
The remaining three terms are estimated by quantities independent of
$u\in\mathbb{K}$. First, decompose
\begin{equation*}
    \mathrm{I}(t)=e^{-\mathcal{A}_\mu t}\big(R_\mu u-E_\mu R_0u\big)
    +\Theta_\mu(t)R_0u .
\end{equation*}
The functions $R_\mu u$ and $E_\mu R_0u$ coincide on $\Omega_\mu\cap\Omega_0$
and are both bounded by $K^{2}C_\infty$ in $L^{\infty}$, so by (H2)
\begin{equation*}
    \zeta(\mu):=\sup_{u\in\mathbb{K}}
    \big\|R_\mu u-E_\mu R_0u\big\|_{L^{2}(\Omega_\mu)}
    \le 2K^{2}C_\infty\,|\Omega_\mu\setminus\Omega_0|^{1/2}
    \longrightarrow0 ,
\end{equation*}
and this bound is uniform in $t$ as well, $e^{-\mathcal{A}_\mu t}$ being a
contraction. For the second part put
$\sigma_\mu(t):=\sup_{z\in R_0\mathbb{K}}\|\Theta_\mu(t)z\|_{L^{2}(\Omega_\mu)}$;
by the elementary fact above, applied to $\Theta_\mu(t)$ on the compact set
$R_0\mathbb{K}$,
\begin{equation}\label{eq:sigma_pointwise}
    \sigma_\mu(t)\longrightarrow0\quad\text{for each fixed }t\in[0,T],
    \qquad
    \sigma_\mu(t)\le2K\sup_{u\in\mathbb{K}}\|R_0u\|_{L^{2}(\Omega_0)}<\infty .
\end{equation}
Second, by Lemma \ref{lem:commutation} applied with the constant $R$ of
\eqref{eq:orbit_linfty}, $\|\mathrm{III}(t)\|_{L^{2}(\Omega_\mu)}\le T\eta(\mu)$
with $\eta(\mu)\to0$ independently of $u$ and of $s$. Third, putting
$\omega_\mu(\sigma):=\sup_{g\in\mathcal{G}}\|\Theta_\mu(\sigma)g\|
_{L^{2}(\Omega_\mu)}$, the same elementary fact applied on the compact set
$\mathcal{G}$ gives $\omega_\mu(\sigma)\to0$ for each fixed $\sigma$, with
$\omega_\mu\le2K\sup_{g\in\mathcal{G}}\|g\|_{L^{2}(\Omega_0)}$ uniformly in
$\mu$ and $\sigma$; hence, by dominated convergence on $[0,T]$,
\begin{equation*}
    \|\mathrm{IV}(t)\|_{L^{2}(\Omega_\mu)}
    \le\int_0^{t}\omega_\mu(\sigma)\,d\sigma
    \le\iota(\mu):=\int_0^{T}\omega_\mu(\sigma)\,d\sigma\longrightarrow0 ,
\end{equation*}
again uniformly in $t\in[0,T]$ and in $u\in\mathbb{K}$.

Together with \eqref{eq:gronwall_kernel}, the bounds just obtained give
$\phi(t)\le g(t)+\ell\int_0^{t}\phi(s)\,ds$ with
\begin{equation*}
    g(t):=\zeta(\mu)+\sigma_\mu(t)+T\eta(\mu)+\iota(\mu) .
\end{equation*}
Gronwall's inequality in integrated form --- if $\phi\le g+\ell\int_0^{t}\phi$
then $\phi(t)\le g(t)+\ell\int_0^{t}e^{\ell(t-s)}g(s)\,ds$, as one sees by
applying $\big(e^{-\ell t}\int_0^{t}\phi\big)'\le e^{-\ell t}g$ --- therefore
yields
\begin{equation}\label{eq:pointwise_flow}
    \phi(T)\;\le\;\big(1+\ell Te^{\ell T}\big)
    \big(\zeta(\mu)+T\eta(\mu)+\iota(\mu)\big)
    +\sigma_\mu(T)+\ell e^{\ell T}\!\int_0^{T}\!\sigma_\mu(s)\,ds .
\end{equation}
Every term on the right tends to $0$ as $\mu\to0$: the first because
$\zeta(\mu),\eta(\mu),\iota(\mu)\to0$, the second by
\eqref{eq:sigma_pointwise} at the single time $T$, and the third by
\eqref{eq:sigma_pointwise} and dominated convergence on $[0,T]$. Since none of
these bounds depends on $u\in\mathbb{K}$, this proves
\eqref{eq:unif_flow}.

Note that the only term depending on $t$ in a non-uniform way is $\sigma_\mu$,
and that it is consumed by \eqref{eq:pointwise_flow} through its value at $T$
and its integral over $[0,T]$. This is why the integrated form of Gronwall's
inequality is used: the convergence \eqref{eq:semigroup_conv} is available at
each fixed time but not uniformly on $[0,T]$, so no bound on
$\sup_{t\le T}\sigma_\mu(t)$ is at our disposal.
\eproof

\subsection{Upper semicontinuity}

\begin{definition}[Upper semicontinuity]\label{def:uppersemicontinuity}
The family $\{\mathscr{A}_\mu\}$ is upper semicontinuous at $\mu=0$ in
$L^{2}$ if
\begin{equation*}
    \lim_{\mu\to0}\operatorname{dist}_{L^{2}(\Omega_\mu)}
    \big(\mathscr{A}_\mu,\,E_\mu\mathscr{A}_0\big)=0,
    \qquad
    \operatorname{dist}_{Z}(X,Y)=\sup_{x\in X}\inf_{y\in Y}\|x-y\|_{Z} .
\end{equation*}
\end{definition}

\begin{theorem}[Upper semicontinuity of the attractors]
\label{thm:upper_semicontinuity}
Assume (H1), (H2) and (H3). Then the family $\{\mathscr{A}_\mu\}$ is upper
semicontinuous at $\mu=0$ in the sense of Definition
\ref{def:uppersemicontinuity}.
\end{theorem}

\proof
Let $\epsilon>0$. The set $B:=\bigcup_{\mu>0}F_\mu\mathscr{A}_\mu$ is bounded in
$H^{1}(\Omega_0)$, by $C\|\psi\|_{W^{1,\infty}}K\rho_{H^1}$, and $\mathscr{A}_0$
attracts bounded sets of $X_1=H^{1}(\Omega_0)$; since
$H^{1}(\Omega_0)\hookrightarrow L^{2}(\Omega_0)$, there is therefore
$T=T(\epsilon)>0$ with
\begin{equation}\label{eq:step1}
    \operatorname{dist}_{L^{2}(\Omega_0)}
    \big(S_0(T)B,\ \mathscr{A}_0\big)<\frac{\epsilon}{2K} ,
\end{equation}
and hence, $E_\mu$ being bounded by $K$ on $L^{2}$, \eqref{eq:step1} gives
\begin{equation}\label{eq:step1b}
    \operatorname{dist}_{L^{2}(\Omega_\mu)}
    \big(E_\mu S_0(T)B,\ E_\mu\mathscr{A}_0\big)<\frac{\epsilon}{2} .
\end{equation}
We fix this $T$ for the remainder of the proof.

Next we make $\mu$ small. For $v\in\mathscr{A}_\mu$ the function
$u:=\psi\,\mathfrak{E}_\mu v$ belongs to $\mathbb{K}$ and satisfies $R_\mu u=v$
and $R_0u=(\psi\,\mathfrak{E}_\mu v)|_{\Omega_0}=F_\mu v$, so Lemma
\ref{lem:unif_conv}, applied at the time $T$ just fixed, provides $\mu_1>0$ such
that
\begin{equation}\label{eq:step2}
    \sup_{v\in\mathscr{A}_\mu}
    \big\|S_\mu(T)v-E_\mu S_0(T)F_\mu v\big\|_{L^{2}(\Omega_\mu)}
    <\frac{\epsilon}{2},
    \qquad \mu<\mu_1 .
\end{equation}

Let now $\mu<\mu_1$ and $w\in\mathscr{A}_\mu$. By strict invariance there is
$v\in\mathscr{A}_\mu$ with $w=S_\mu(T)v$, and $v_0:=F_\mu v\in B$, so that
\begin{equation*}
    \operatorname{dist}_{L^{2}(\Omega_\mu)}\big(w,E_\mu\mathscr{A}_0\big)
    \le\big\|S_\mu(T)v-E_\mu S_0(T)v_0\big\|_{L^{2}(\Omega_\mu)}
    +\operatorname{dist}_{L^{2}(\Omega_\mu)}
     \big(E_\mu S_0(T)v_0,\ E_\mu\mathscr{A}_0\big)
    <\frac{\epsilon}{2}+\frac{\epsilon}{2} ,
\end{equation*}
the first term by \eqref{eq:step2} and the second by \eqref{eq:step1b}. Taking
the supremum over $w\in\mathscr{A}_\mu$ concludes the proof.
\eproof

\begin{cor}[Upper semicontinuity in $H^{s}$, $s<1$]\label{cor:semicontinuity_Hs}
Under the assumptions of Theorem \ref{thm:upper_semicontinuity}, for every
$s\in[0,1)$,
\begin{equation*}
    \lim_{\mu\to0}\operatorname{dist}_{H^{s}(\Omega_\mu)}
    \big(\mathscr{A}_\mu,\,E_\mu\mathscr{A}_0\big)=0 .
\end{equation*}
\end{cor}

\proof
By Lemma \ref{lem:uniform_interpolation}, $H^{s}(\Omega_\mu)
=[L^{2}(\Omega_\mu),H^{1}(\Omega_\mu)]_{s}$ with constants independent of
$\mu$, and the complex interpolation method gives
$\|z\|_{H^{s}(\Omega_\mu)}\le C\|z\|^{1-s}_{L^{2}(\Omega_\mu)}
\|z\|^{s}_{H^{1}(\Omega_\mu)}$ with $C=C(n,s,K)$. If $w\in\mathscr{A}_\mu$ and
$z_0\in\mathscr{A}_0$, then $\|w-E_\mu z_0\|_{H^{1}(\Omega_\mu)}\le
(1+K)\rho_{H^1}$ by Proposition \ref{prop:absorbing_X1}. Hence
\begin{equation*}
    \operatorname{dist}_{H^{s}(\Omega_\mu)}
    \big(\mathscr{A}_\mu,E_\mu\mathscr{A}_0\big)
    \le C\big((1+K)\rho_{H^1}\big)^{s}
    \Big(\operatorname{dist}_{L^{2}(\Omega_\mu)}
    \big(\mathscr{A}_\mu,E_\mu\mathscr{A}_0\big)\Big)^{1-s},
\end{equation*}
and the right-hand side tends to $0$ by Theorem
\ref{thm:upper_semicontinuity}.
\eproof

\begin{rem}[The endpoint $s=1$]\label{rem:open_s1}
Corollary \ref{cor:semicontinuity_Hs} does not reach $s=1$, and this is not an
artefact of the argument. Upper semicontinuity in $H^{1}(\Omega_\mu)$ would
require the convergence \eqref{eq:unif_flow} in the $H^{1}$ topology, hence the
precompactness of the set $\mathbb{K}$ of \eqref{eq:K_set} in
$H^{1}(\mathbb{R}^{n})$ rather than in $L^{2}(\mathbb{R}^{n})$. Such a statement
lies outside the present framework, for the following reason.

The natural attempt is to bound the attractors uniformly in
$H^{1+\tau}(\Omega_\mu)$ for some $\tau>0$. What the parabolic smoothing
provides is a uniform bound in the fractional power space $X_{1+\tau}$, and, as
explained in Remark \ref{rem:scope_identifications}, on a Jones domain with
merely bounded measurable coefficients this space is not known to embed into any
$H^{1+\tau'}(\Omega_\mu)$. Since the extension operators are bounded on the
classical Sobolev scale and not on the scale of fractional powers, the bound
cannot be transported to $H^{1}(\mathbb{R}^{n})$, and the compactness of
$\mathbb{K}$ in that space is lost. A different strategy, which avoids the
fixed ambient space altogether, is outlined in Section
\ref{sec:concluding_remarks}; it is, in substance, the route followed in
\cite{ArrietaCarvalho2004}, where continuity in $H^{1}$ is obtained on Lipschitz
families from the spectral convergence of the Neumann Laplacian.
\end{rem}

\section{Concluding Remarks and Future Perspectives}
\label{sec:concluding_remarks}

In this paper we have established the robust asymptotic behaviour of a class of
semilinear parabolic equations on a family of varying non-smooth domains. The
geometric hypotheses are only that the family satisfies a uniform Jones
condition, that its volumes are uniformly bounded, and that
$|\Omega_\mu\triangle\Omega_0|\to0$; containment in a common bounded region is
not assumed, being a consequence of the first two by Lemma
\ref{lem:uniform_diameter}. Nor is any smoothness of the boundaries assumed:
they need not be locally graphs, need not be Lipschitz, and may be
nonrectifiable, the class of uniform Jones domains containing, for instance,
domains with fractal boundary. The condition is nevertheless not vacuous --- it
is a scale-invariant accessibility requirement, asking that any two points be
joined inside the domain by a curve of comparable length keeping a proportional
distance from the boundary, and it excludes outward cusps and slits --- so what
is given up is boundary regularity in the analytic sense, not geometric
connectedness.

Under these hypotheses we obtained two facts of different natures, which it is
worth keeping apart. The first is a statement about a set: the global attractors
are uniformly bounded in $L^{\infty}(\Omega_\mu)$, by a Stampacchia truncation
argument occupying half a page, with the explicit constant
$L=\sqrt{2C_f/\beta}$, which depends only on the nonlinearity and on $\beta$,
neither on $n$ nor on the geometry of the domains. The second is a statement
about the dynamics, strictly stronger and of independent interest: bounded sets
of $L^{q}(\Omega_\mu)$ are absorbed by a fixed ball of $L^{\infty}(\Omega_\mu)$
within a time depending only on their radius, uniformly in $\mu$. Proved by a
Moser--Alikakos iteration, it makes no use of invariance and applies to
arbitrary trajectories, and it is there that the uniform Sobolev and
Gagliardo--Nirenberg constants of the earlier sections are needed. From the
first of the two we deduced that the attractors are upper semicontinuous at
$\mu=0$ in $L^{2}(\Omega_\mu)$, hence in $H^{s}(\Omega_\mu)$ for every $s<1$;
the second is not used there, its interest lying rather in being the uniform
smoothing step on which any assertion about exponential attraction must rest.

A key asset of the approach is that all embedding, interpolation and
regularising constants are independent of $\mu$. The roughness of the domains is
not measured by $\mu$, but held under uniform control by (H1), whose Jones
constants $(\delta_0,\epsilon)$ do not depend on $\mu$; it is this uniformity,
and not any smallness, that the estimates exploit. Its effect is visible in
Section \ref{sec:upper_semicontinuity}: the truncation bound of Proposition
\ref{prop:truncation_linfty} together with the absorption rate of Proposition
\ref{cor:absorption_rate} confine every orbit entering the analysis to a fixed
ball of $L^{\infty}$, on which the nonlinearity is Lipschitz with a
$\mu$-independent constant, and it is this that reduces the whole continuity
analysis to elementary estimates in $L^{2}$.

The results and methods developed here open several lines of investigation,
which we plan to address in forthcoming work. We describe four: the endpoint
$s=1$, intrinsic topological convergence, rates of convergence, and
boundary-driven dynamics.

\subsection*{The Endpoint $s=1$ and Convergence of Spectral Structures}

Corollary \ref{cor:semicontinuity_Hs} gives upper semicontinuity in
$H^{s}(\Omega_\mu)$ for every $s<1$, and Remark \ref{rem:open_s1} explains why
the endpoint $s=1$ is not reachable by the arguments of Section
\ref{sec:upper_semicontinuity}: it would require precompactness of the set
$\mathbb{K}$ of \eqref{eq:K_set} in $H^{1}(\mathbb{R}^{n})$, and the parabolic
smoothing produces regularity in the fractional power scale, which the
extension operators cannot transport.

There is, however, a strategy that dispenses with the fixed ambient space
altogether. Instead of extending the attractors to $\mathbb{R}^{n}$, one may
work directly with the convergence of the \emph{spectral structures} of the
operators $\mathcal{A}_\mu$ --- convergence of eigenvalues and of spectral
projections in the sense of Kuwae and Shioya \cite{KuwaeShioya2003} --- for
which the Mosco convergence of Theorem \ref{thm:mosco} is precisely the
required input. The second ingredient would be a uniform Weyl-type lower bound
\begin{equation}\label{eq:weyl}
    \lambda_j(\mu)\ \ge\ c\,j^{2/n},\qquad j\ge1,
\end{equation}
with $c$ independent of $\mu$, which follows from the uniform Sobolev embedding
of Lemma \ref{lem:uniform_sobolev}: the latter yields a Nash inequality with a
$\mu$-independent constant, hence uniform ultracontractivity of the semigroups
$e^{-\mathcal{A}_\mu t}$, hence a uniform bound on
$\operatorname{Tr}e^{-\mathcal{A}_\mu t}$ and therefore on the counting
function. Estimate \eqref{eq:weyl} gives uniform smallness of the
high-frequency part of the attractors, and, combined with the convergence of
the low-frequency spectral projections, should supply the compactness that is
missing in $H^{1}$. We note that this route would also make good on a
computation of independent interest --- the uniform ultracontractivity of the
linear semigroups on a uniform Jones family --- which the methods of Section
\ref{sec:uniform_linfty} come close to but do not require.

\subsection*{Intrinsic Convergence and Gromov-Hausdorff Distance}

In this work, the proximity of the perturbed domains $\Omega_\mu$ to the reference domain $\Omega_0$ is measured extrinsically via the Lebesgue measure of their symmetric difference. While this condition naturally interfaces with the extension-restriction framework and standard upper semicontinuity theorems, it is strictly tied to the ambient spatial coordinates. Consequently, it fails to recognise the dynamical equivalence of domains that are identical up to a rigid transformation (e.g., spatial translations), where the symmetric difference may be large despite the dynamics being isometrically identical.

The simplest instance is already within reach. If $T_\mu$ are rigid motions of
$\mathbb{R}^{n}$ with $|T_\mu(\Omega_\mu)\triangle\Omega_0|\to0$, then
$u\mapsto u\circ T_\mu^{-1}$ is an isometry of $L^{2}(\Omega_\mu)$ onto
$L^{2}(T_\mu(\Omega_\mu))$ and of the corresponding $H^{1}$ spaces, so that
Theorem \ref{thm:upper_semicontinuity} applies verbatim to the moved family and
yields upper semicontinuity with the connecting maps precomposed with those
isometries. One caveat must be recorded: the operator is transported along with
the domain only if the coefficient matrix is transported as well, the induced
matrix on $T_\mu(\Omega_\mu)$ being $R_\mu A(T_\mu^{-1}x)R_\mu^{\mathsf{T}}$
with $R_\mu$ the rotational part of $T_\mu$. For a fixed field $A$ this is an
additional hypothesis, automatically satisfied when $A$ is constant.

The genuinely intrinsic problem is harder. As detailed in recent literature
concerning the Gromov--Hausdorff distance of global attractors under domain
perturbations (see, e.g., \cite{lee2020gromov}), one would compare
$\mathscr{A}_\mu$ and $\mathscr{A}_0$ as abstract metric spaces, with no
ambient rigid motion available and the connecting maps built out of
$\epsilon$-isometries. The uniform $L^{\infty}$ bound of Theorem
\ref{thm:uniform_linfty} and the uniform $H^{1}$ bound on the attractors,
together with the uniform Jones extension operators, provide part of what such
a framework requires; what is missing is a uniform modulus of continuity for
the comparison maps, which the present hypotheses do not supply.

\subsection*{Rates of Convergence and Exponential Attraction}

Theorem \ref{thm:upper_semicontinuity} and Corollary \ref{cor:semicontinuity_Hs}
are purely qualitative: they assert that the semidistance vanishes, without any
information on how fast. The reason is structural rather than technical. In Step
1 of the proof of Theorem \ref{thm:upper_semicontinuity} the time $T$ is chosen
using only the fact that $\mathscr{A}_0$ attracts bounded sets, a statement that
carries no rate; everything quantitative in the argument is concentrated in the
other half, where the geometry enters.

It is worth making explicit what that quantitative half provides. Writing
$\varepsilon:=|\Omega_\mu\triangle\Omega_0|$, the almost-commutation estimate of
Lemma \ref{lem:commutation} and the convergence of the initial data are both of
order $\varepsilon^{1/2}$, and the Gronwall argument of Lemma
\ref{lem:unif_conv} amplifies them by $e^{\ell T}$, where $\ell$ is the Lipschitz
constant of $f$ on the $L^{\infty}$ ball of radius $KR$, with
$R=\max\{L,KC_\infty\}$ as in \eqref{eq:orbit_linfty}. Thus
\begin{equation}\label{eq:flow_rate}
    \sup_{v\in\mathscr{A}_\mu}
    \big\|S_\mu(T)v-E_\mu S_0(T)F_\mu v\big\|_{L^{2}(\Omega_\mu)}
    \ \le\ C\,e^{\ell T}\varepsilon^{1/2} .
\end{equation}

Suppose now that the limit attractor were \emph{exponentially} attracting, that
is, that there were constants $M_0>0$ and $\kappa>0$ with
\begin{equation}\label{eq:exp_attraction}
    \operatorname{dist}_{L^{2}(\Omega_0)}\big(S_0(t)B,\mathscr{A}_0\big)
    \ \le\ M_0e^{-\kappa t},\qquad t\ge0,
\end{equation}
for the bounded set $B=\bigcup_{\mu>0}F_\mu\mathscr{A}_\mu$. Combining
\eqref{eq:flow_rate} with \eqref{eq:exp_attraction} and optimising in $T$ ---
the two terms balance at $T=\frac{1}{2(\ell+\kappa)}\log(1/\varepsilon)$ ---
would yield the H\"older rate
\begin{equation}\label{eq:holder_rate}
    \operatorname{dist}_{L^{2}(\Omega_\mu)}
    \big(\mathscr{A}_\mu,E_\mu\mathscr{A}_0\big)
    \ \le\ C\,\varepsilon^{\,\theta},
    \qquad \theta=\frac{\kappa}{2(\ell+\kappa)} ,
\end{equation}
and, by interpolation as in Corollary \ref{cor:semicontinuity_Hs}, the rate
$\varepsilon^{(1-s)\theta}$ in $H^{s}(\Omega_\mu)$ for $s<1$.

The missing ingredient is therefore \eqref{eq:exp_attraction}, and it is of a
different nature from everything used in this paper. Two classical routes lead
to it. The first is to assume a gradient structure for the limit problem
together with hyperbolicity of its equilibria, in the spirit of Hale and Raugel
\cite{HaleRaugel1992}; this is a hypothesis on the unperturbed problem alone,
and would be legitimate here, although it runs against the deliberate choice,
recorded in Remark \ref{rem:non_gradient_structure}, of not relying on a
Lyapunov functional. The second, and to us the more attractive, is to construct
\emph{exponential attractors} $\mathscr{M}_\mu\supset\mathscr{A}_\mu$ uniformly
in $\mu$, in the sense of Eden, Foias, Nicolaenko and Temam
\cite{EdenFoiasNicolaenkoTemam1994}.
What such a construction requires is a smoothing, or squeezing, property for
\emph{differences} of solutions, together with uniform bounds on the
trajectories. The second is available here, uniformly in $\mu$, by Theorem
\ref{thm:uniform_linfty}, which in addition makes the nonlinearity globally
Lipschitz on the absorbing set; the first should follow from the analytic
smoothing of $e^{-\mathcal{A}_\mu t}$ on the fractional scale, by the argument
used in the proof of Theorem \ref{thm:global_attractor} applied to the
difference of two solutions rather than to a single one, but we have not
carried this out. It should be stressed that an exponential attractor attracts
exponentially to itself and not to $\mathscr{A}_\mu$ --- the possibility that
the global attractor attracts arbitrarily slowly is precisely what motivates
the notion --- so that \eqref{eq:exp_attraction} would not follow from such a
construction alone. A uniform family $\{\mathscr{M}_\mu\}$ would, however,
yield the analogue of \eqref{eq:holder_rate} with $\mathscr{A}_0$ replaced by
$\mathscr{M}_0$, which is already a statement of some interest.

\subsection*{Boundary Perturbations: From Robin to Nonlinear Dissipation}

Another physically relevant extension is to consider the evolution problem under boundary interactions, starting with homogeneous Robin boundary conditions:
\begin{equation}\label{eq:robin_future}
    \frac{\partial u}{\partial \nu_A} + b(x) u = 0 \quad \text{on } \partial\Omega_\mu,
\end{equation}
where $b \in L^\infty(\partial\Omega_\mu)$ is a positive weight (the letter
$\beta$ being reserved for the dissipativity constant of Hypothesis
\ref{hyp:nonlinear_dissipativity}). Analytically, transitioning from Neumann to Robin conditions introduces the explicit appearance of boundary integrals $\int_{\partial\Omega_\mu} b\, u v \, d\sigma_\mu$ in the weak formulation. Crucially, the volume convergence of the domains does not generally imply the convergence of the surface measures $\sigma_\mu$ when the boundary exhibits high oscillations or fractal-like roughness. Investigating how to redefine Mosco convergence for boundary-weighted forms under the uniform Jones property is a critical next step.

Beyond linear conditions, we intend to study problems featuring localised nonlinearities on the boundary:
\begin{equation}\label{eq:nonlinear_boundary}
    \frac{\partial u}{\partial \nu_A} + u = g(u) \quad \text{on } \partial\Omega_\mu,
\end{equation}
where $g$ satisfies a subcritical growth condition. In this scenario, the dissipative mechanism itself is concentrated on a highly irregular geometric interface. Adapting the bootstrap iterations presented in this paper to boundary integrals will require stricter geometric assumptions on $\partial \Omega$---such as uniform Lipschitz or Ahlfors regularity of the boundary measures---to ensure the uniform validity of fractional trace embeddings. This interplay between fractional power spaces and varying boundary trace maps will require an intricate synthesis of geometric measure theory and infinite-dimensional dynamical systems.

\subsection*{Pull-back Methods on a Rough Base Domain}

A different line of attack on domain perturbation, going back to Henry
\cite{Henry1981} and used in
\cite{PereiraPereira2007,BarbosaPereiraPereira2016} among many others, presents
the perturbed domains as images $\Omega_\mu=h_\mu(\Omega_0)$ of a fixed domain
and pulls the problem back by $v=u\circ h_\mu$. The hypotheses adopted here
supply no such family of maps, and this is the precise sense in which the two
frameworks are not comparable: a uniform Jones family converging in measure
need not be a bi-Lipschitz deformation of $\Omega_0$ at all --- bi-Lipschitz
maps preserve Hausdorff dimension, so a family whose boundaries have varying
dimension is already excluded --- nor need its members be homeomorphic to one
another.

It is worth observing, however, that the pull-back is compatible with the class
of operators considered here in a way that it is not with smoother ones. If $h$
is bi-Lipschitz then $u\mapsto u\circ h$ is an isomorphism of
$H^{1}(h(\Omega_0))$ onto $H^{1}(\Omega_0)$, and
$-\operatorname{div}(A\nabla\,\cdot\,)$ is carried into
$-J^{-1}\operatorname{div}(J\widetilde{A}\nabla\,\cdot\,)$, where
$\widetilde{A}=(Dh)^{-1}A(h)(Dh)^{-\mathsf{T}}$ and $J=|\det Dh|$ are again
bounded and measurable and $\widetilde{A}$ is uniformly elliptic, the
homogeneous Neumann condition being transported to the conormal condition for
$\widetilde{A}$. Bounded measurable coefficients are preserved by such a change
of variables, whereas smooth ones are not; and nothing in the pull-back
requires the base domain to be regular.

This suggests a programme that we have not pursued: to take $\Omega_0$ itself
to be a Jones domain, as rough as one pleases, and to study perturbations
$\Omega_\mu=h_\mu(\Omega_0)$ given by bi-Lipschitz maps. The uniform apparatus
of Section \ref{sec:preliminaries} would then be required for the single domain
$\Omega_0$, and what remained would be a perturbation problem on a fixed phase
space, with the attendant gains: continuity rather than mere upper
semicontinuity, in the $H^{1}$ topology, with rates, and --- when the maps
depend smoothly on the parameter --- the differential calculus with respect to
the perturbation which is the chief merit of the pull-back and which has no
counterpart in the framework of varying spaces. What is not clear to us is how
much of that calculus survives when $\partial\Omega_0$ is not rectifiable,
since the shape-derivative formulas are usually written as boundary integrals.a

 \subsection*{Declaration of Generative AI and AI-assisted Technologies
in the Writing Process}

During the preparation of this work the author used Google's Gemini to assist
with English language refinement, \LaTeX{} formatting and the structural
organisation of an earlier version of the manuscript, and Anthropic's Claude in
the subsequent revision, where it was used to analyse a referee report, to
identify errors and gaps in the arguments, and to draft revised statements and
proofs. All mathematical content was verified, corrected and, where necessary,
rewritten by the author, who takes full responsibility for the content of the
publication.

\bibliographystyle{plain} 
\bibliography{referencias}

@article{Jones1981,
  title={Quasiconformal mappings and extendability of functions in {S}obolev spaces},
  author={Jones, Peter W},
  journal={Acta Mathematica},
  volume={147},
  pages={71--88},
  year={1981},
  publisher={Springer}
}

@article{foias_prodi_1967,
  title={Sur le comportement global des solutions non-stationnaires des {\'e}quations de {N}avier-{S}tokes en dimension 2},
  author={Foias, Ciprian and Prodi, Giovanni},
  journal={Rendiconti del Seminario Matematico della Universit{\`a} di Padova},
  volume={39},
  pages={1--34},
  year={1967},
  publisher={Seminario Matematico of the University of Padua}
}

@book{EdenFoiasNicolaenkoTemam1994,
  author    = {Eden, A. and Foias, C. and Nicolaenko, B. and Temam, R.},
  title     = {Exponential Attractors for Dissipative Evolution Equations},
  series    = {RAM: Research in Applied Mathematics},
  volume    = {37},
  publisher = {Masson / John Wiley},
  address   = {Paris / Chichester},
  year      = {1994}
}

@book{temam1988,
  title={Infinite-Dimensional Dynamical Systems in Mechanics and Physics},
  author={Temam, Roger},
  volume={68},
  series={Applied Mathematical Sciences},
  year={1988},
  publisher={Springer-Verlag},
  address={New York}
}

@article{ArrietaCarvalhoRodriguezBernal2000,
   author  = {Arrieta, J. M. and Carvalho, A. N. and Rodr{\'\i}guez-Bernal, A.},
   title   = {Attractors of parabolic problems with nonlinear boundary
              conditions. Uniform bounds},
   journal = {Comm. Partial Differential Equations},
   volume  = {25},
   number  = {1-2},
   pages   = {1--37},
   year    = {2000}
 }

@unpublished{Tavares2026,
  author = {Tavares-Lima, Elaine A.  and Lorenzi, Bianca  and  Pereira, Marcone C.},
  title  = {Attractor Continuity for Semilinear Parabolic Equations on Thin Domains with Degenerating Outward Peaks},
  note   = {Preprint, arXiv:2602.21358 [math.AP]},
year   = {2026},
}

@book{Yagi2010,
  title={Abstract Parabolic Evolution Equations and their Applications},
  author={Yagi, Atsushi},
  year={2010},
  publisher={Springer},
  address={Berlin, Heidelberg}
}

@book{Henry1981,
  title={Geometric Theory of Semilinear Parabolic Equations},
  author={Henry, Daniel},
  volume={840},
  series={Lecture Notes in Mathematics},
  year={1981},
  publisher={Springer},
  address={Berlin, New York}
}

@article{ArrietaCarvalho2004,
  title={Spectral convergence and nonlinear dynamics of reaction-diffusion equations under perturbations of the domain},
  author={Arrieta, Jos{\'e} M and Carvalho, Alexandre N},
  journal={Journal of Differential Equations},
  volume={199},
  number={1},
  pages={143--178},
  year={2004},
  publisher={Elsevier}
}

@article{Miklavicic1985,
  author  = {Miklav\v{c}i\v{c}, Milan},
  title   = {Stability for semilinear equations with noninvertible linear operator},
  journal = {Pacific J. Math.},
  volume  = {118},
  number  = {1},
  pages   = {199--214},
  year    = {1985}
}

@article{HaleRaugel1992,
  author  = {Hale, Jack K. and Raugel, Genevi{\`e}ve},
  title   = {Reaction-diffusion equation on thin domains},
  journal = {Journal de Math{\'e}matiques Pures et Appliqu{\'e}es},
  volume  = {71},
  number  = {1},
  pages   = {33--95},
  year    = {1992},
  publisher = {Elsevier}
}

@book{MazyaPoborchi,
  author    = {Maz'ya, Vladimir G. and Poborchi, Sergey V.},
  title     = {Differentiable Functions on Bad Domains},
  year      = {1997},
  publisher = {World Scientific},
  address   = {Singapore}
}

@book{Attouch1984,
  author    = {Attouch, Hedy},
  title     = {Variational Convergence for Functions and Operators},
  series    = {Applicable Mathematics Series},
  publisher = {Pitman Advanced Publishing Program},
  address   = {Boston},
  year      = {1984}
}

@article{Daners2008,
  author  = {Daners, Daniel},
  title   = {Domain perturbation for linear and semi-linear boundary value problems},
  journal = {Handbook of Differential Equations: Stationary Partial Differential Equations},
  volume  = {6},
  pages   = {1--81},
  year    = {2008},
  publisher = {Elsevier}
}

@article{lee2020gromov,
  title={{G}romov-{H}ausdorff stability of global attractors of reaction diffusion equations under perturbations of the domain},
  author={Lee, Jihoon and Nguyen, Nguyen and Toi, Vu Manh},
  journal={Journal of Differential Equations},
  volume={269},
  number={1},
  pages={125--147},
  year={2020},
  publisher={Elsevier}
}

@article{Rogers2006,
   author  = {Rogers, Luke G.},
   title   = {Degree-independent {S}obolev extension on locally uniform domains},
   journal = {J. Funct. Anal.},
   volume  = {235},
   number  = {2},
   pages   = {619--665},
   year    = {2006}
 }

@article{KuwaeShioya2003,
  author  = {Kuwae, K. and Shioya, T.},
  title   = {Convergence of spectral structures: a functional analytic theory
             and its applications to spectral geometry},
  journal = {Comm. Anal. Geom.},
  volume  = {11},
  number  = {4},
  pages   = {599--673},
  year    = {2003}
}

@book{LionsMagenes1972,
   author    = {Lions, J.-L. and Magenes, E.},
   title     = {Non-Homogeneous Boundary Value Problems and Applications, Vol.~I},
   publisher = {Springer-Verlag},
   address   = {Berlin--Heidelberg--New York},
   year      = {1972}
 }

@book{Triebel1978,
   author    = {Triebel, H.},
   title     = {Interpolation Theory, Function Spaces, Differential Operators},
   publisher = {North-Holland},
   address   = {Amsterdam},
   year      = {1978}
 }

@book{Amann1995,
   author    = {Amann, H.},
   title     = {Linear and Quasilinear Parabolic Problems, Vol.~I: Abstract Linear Theory},
   publisher = {Birkh\"auser},
   address   = {Basel},
   year      = {1995}
 }

@article{BarbosaPereiraPereira2016,
  author  = {Barbosa, Pricila S. and Pereira, Ant{\^o}nio L. and Pereira, Marcone C.},
  title   = {Continuity of attractors for a family of $\mathcal{C}^1$ perturbations of the square},
  journal = {Annali di Matematica Pura ed Applicata},
  volume  = {196},
  pages   = {1--34},
  year    = {2016}
}

@article{AragaoPereiraPereira2014,
  author  = {Arag{\~a}o, Gleiciane S. and Pereira, Ant{\^o}nio L. and Pereira, Marcone C.},
  title   = {Attractors for a Nonlinear Parabolic Problem with Terms Concentrating on the Boundary},
  journal = {Journal of Dynamics and Differential Equations},
  volume  = {26},
  pages   = {871--888},
  year    = {2014}
}

@article{PereiraPereira2007,
  author  = {Pereira, Ant{\^o}nio L. and Pereira, Marcone C.},
  title   = {Continuity of attractors for a reaction-diffusion problem with nonlinear boundary conditions with respect to variations of the domain},
  journal = {Journal of Differential Equations},
  volume  = {239},
  pages   = {343--370},
  year    = {2007}
}

@article{PereiraOliva2002,
  author  = {Pereira, Ant{\^o}nio L. and Oliva, S{\'e}rgio M.},
  title   = {Attractors for parabolic problems with nonlinear boundary conditions in fractional power spaces},
  journal = {Dynamics of Continuous, Discrete and Impulsive Systems},
  volume  = {9},
  number  = {4},
  pages   = {551--562},
  year    = {2002}
}

@article{CarvalhoOlivaPereiraRodriguez1997,
  author  = {Carvalho, Alexandre N. and Oliva, S{\'e}rgio M. and Pereira, Ant{\^o}nio L. and Rodr{\'\i}guez-Bernal, An{\'\i}bal},
  title   = {Attractors For Parabolic Problems With Nonlinear Boundary Conditions},
  journal = {Journal of Mathematical Analysis and Applications},
  volume  = {207},
  number  = {2},
  pages   = {409--461},
  year    = {1997}
}

@misc{Pereira_arxiv2024,
  author        = {Pereira, Ant{\^o}nio L.},
  title         = {Continuity of attractors for a highly oscillatory family of perturbations of the square},
  year          = {2024},
  eprint        = {2408.07204},
  archivePrefix = {arXiv},
  note          = {Preprint, \href{https://arxiv.org/abs/2408.07204}{arXiv:2408.07204}}
}

\end{document}